\documentclass[11pt,reqno]{article}

\usepackage[utf8]{inputenc}
\usepackage[T1]{fontenc}
\usepackage{amsmath,amssymb}
\usepackage{amsthm}
\usepackage{helvet}
\usepackage{graphicx}
\usepackage{multicol}
\usepackage[bottom]{footmisc}
\usepackage{booktabs}
\usepackage{newtxtext}
\usepackage[varvw]{newtxmath}
\usepackage{mathbbol}
\usepackage{hyperref}
\usepackage{fullpage}

\usepackage{todonotes}

\newtheorem{theorem}{Theorem}[section]
\newtheorem{proposition}[theorem]{Proposition}

\newtheorem{corollary}[theorem]{Corollary}
\theoremstyle{definition}

\theoremstyle{remark}
\newtheorem{remark}[theorem]{Remark}

\makeindex

\begin{document}

\title{Almost-Hermitian and Hermitian metrics}
\author{Daniele Angella\thanks{Dipartimento di Matematica e Informatica, Università di Firenze, viale Morgagni 67/A, 50134 Firenze, Italy, \href{mailto:daniele.angella@unifi.it}{daniele.angella@unifi.it}}}

\maketitle

\begin{flushright}
\begin{footnotesize}
\begin{minipage}{.6\textwidth}
{\itshape La scienza -- così come la legge -- nasce sempre come esigenza di tutela e di liberazione dell'uomo, ma è facile si traduca in un nuovo strumento di oppressione. La tecnica -- così come la legge -- può dunque essere usata come strumento di liberazione se riusciamo ogni volta a comprendere i bisogni reali cui si deve rispondere, evitando di presumere o di accettare che la scienza e la legge servano a rispondere ai bisogni dei tecnici o della società che li delega.}
\smallskip
\begin{flushright}
Franco Basaglia, Franca Ongaro Basaglia, \emph{Crimini di pace}
\end{flushright}
\end{minipage}
\end{footnotesize}
\bigskip
\end{flushright}

\abstract{
We describe three problems in Hermitian geometry related to the scalar curvature of the Chern connection, aimed at identifying and constructing canonical metrics. They are chosen as toy examples where different techniques naturally interact: analytic methods for partial differential equations, the notion of momentum map, and variational methods.
}

\section{Introduction and motivation}
\label{sec:introduction}

As illustrated by the foundational Uniformization Theorem, encoding information in canonically attached metrics is a powerful tool for attacking classification problems. Here, ``canonical'' refers to uniqueness up to the corresponding automorphism group.
But how do canonical metrics emerge? Asking for prescribed curvature (e.g.\ constant curvature, in some sense) usually translates into solving partial differential equations on manifolds, which may be solvable when coupled with other constraints on the metric (e.g.\ coupling the Ricci-flat and K\"ahler conditions leads to a complex Monge-Amp\`ere equation, whose solution gives Calabi-Yau metrics). These elliptic equations also have parabolic counterparts, so that canonical metrics may emerge as the limiting behaviour of geometric flows.
Moreover, they can sometimes be interpreted as Euler-Lagrange equations of suitable functionals, whose critical points are precisely the metrics we are looking for: for instance, extremal K\"ahler metrics are defined as critical points of the Calabi functional, given by the $L^2$-norm of the scalar curvature, and include constant scalar curvature K\"ahler and K\"ahler-Einstein metrics. In the K\"ahler setting, the scalar curvature plays a particularly important role: by the foundational work of Fujiki and Donaldson, it arises as a momentum map for an infinite-dimensional Hamiltonian action on the space of K\"ahler structures, and constant scalar curvature K\"ahler metrics arise precisely as its zeroes.

In these lectures, we cover three problems in Hermitian geometry where all these approaches appear in a natural way. They are taken from joint works of the author with Simone Calamai, Francesco Pediconi, Carlo Scarpa, Cristiano Spotti, and Oluwagbenga Joshua Windare \cite{angella-calamai-spotti,angella-calamai-pediconi-spotti,angella-pediconi-scarpa-spotti-windare}.
Why this choice? First of all, because the author knows them well. He is conscious
that these results will not revolutionize complex geometry, but he regards them as
excellent toy examples from a pedagogical point of view, well suited to conveying the
techniques he learned during the preparation of those papers, thanks especially to
his coauthors.

In Section~\ref{sec:preliminaries}, we recall some preliminaries on the Chern connection. In particular, we introduce the Chern scalar curvature, the main character of these notes, and compute its conformal transformation rule together with those of some related geometric quantities.
In Section~\ref{sec:chern-yamabe}, we introduce an analogue of the Yamabe problem, looking for Hermitian metrics in a conformal class with constant scalar curvature with respect to the Chern connection. This problem was originally introduced in \cite{angella-calamai-spotti} and studied as a partial differential equation of Liouville type, which is not variational in general. Motivated by the Fujiki-Donaldson picture, we asked ourselves whether the Chern scalar curvature still admits an interpretation as a momentum map. This was studied in \cite{angella-calamai-pediconi-spotti} in a specific Hermitian non-K\"ahler setting, which we discuss in Section~\ref{sec:momentum}. In that work, it emerged that the Chern scalar curvature itself may not be the most natural quantity to study, but rather a deformation of it involving torsion terms. The corresponding deformed Yamabe-type problems arise as the Euler-Lagrange equations of a functional and were studied in \cite{angella-pediconi-scarpa-spotti-windare} using variational techniques, which we recall in Section~\ref{sec:twisted-yamabe}.
Some of these problems have counterparts or generalizations in the almost-Hermitian setting, which is in fact the natural framework for the momentum map picture; see e.g.\ \cite{delrio-simanca,lejmi-upmeier,barbaro-lejmi}.

Before continuing, the author would like to sadly highlight that all the mathematicians named in these notes (whether associated with a theorem, a geometric object, or a foundational result) studied in the Global North and are male. This reflects how far our field still is from inclusiveness and decolonization, and should motivate us to explore all the culture we have yet to discover.
The author would like to be able to apply the teaching of Basaglia and Ongaro Basaglia, who argue in the same volume cited above that any strategy should first recognize people (all people, not an abstract notion) and their needs as its ultimate aim, rather than the needs of those who practice it or of the institutions that legitimize it.
\section{Preliminaries and notation}
\label{sec:preliminaries}

\subsection{Setting and notation}

In these notes, we work on complex manifolds $X$, always assumed to be connected and, unless otherwise stated, compact. The superscript $n$ in $X^n$ denotes the complex dimension. We will denote by $J$ the associated almost-complex structure, which we usually assume integrable, even if many objects and results can be extended to the non-integrable case.
We usually assume functions and tensors to be smooth, even if the singular setting is also very interesting in some situations. 
We do not assume the existence of K\"ahler metrics or other special structures; instead, we consider general Hermitian metrics. We encode them in the datum of a real positive $(1,1)$-form $\omega$, while the underlying Riemannian metric is $g=\omega(\cdot,J\cdot)$. The tensor $h=g-\sqrt{-1}\,\omega$ then yields a Hermitian structure on the fibres of the holomorphic tangent bundle.
We always integrate with respect to the volume form $\frac{\omega^n}{n!}$ naturally associated with $g$ and the canonical orientation induced by the complex structure. For $1\leq p<+\infty$, the $L^p$-norm of a function $f$ is
$$ \|f\|_{L^p} := \left( \int_X |f|^p\, \frac{\omega^n}{n!}\right)^{\frac{1}{p}} . $$
We denote by the same symbol the induced $L^2$-pairings on forms and endomorphisms, extending the scalar product to tensors in the natural way; in particular, on endomorphisms, $\langle A,B\rangle_{L^2}=\int_X\mathrm{tr}(AB^{t})\,\frac{\omega^n}{n!}$, where $(\cdot)^t$ denotes the $g$-adjoint.

We denote by $\Delta:=dd^*+d^*d$ the Hodge-de Rham (``geometric'') Laplacian with respect to $\omega$, and write $\Delta_\omega$ when we need to avoid ambiguity. Recall that $d^*$ is the formal adjoint of $d$ with respect to the $L^2$-inner product induced by $g$ on the space of global smooth differential forms.
In particular, integration by parts gives $\int_X f\,\Delta u\,\frac{\omega^n}{n!} = \int_X g(df,du)\,\frac{\omega^n}{n!}$ for smooth functions $f$ and $u$ on $X$.
With our sign convention, for any smooth real-valued function $f$ on $X$, at a local maximum point $x_{\max}$ we have $(\Delta f)(x_{\max})\geq 0$. We denote by $\nabla$ the gradient operator on smooth functions induced by $g$, defined by $\nabla f := (df)^\sharp$. Here, $\sharp$ denotes the musical isomorphism induced by $g$, namely $df(X)=g(X,(df)^{\sharp})$ for any vector field $X$, and $\flat$ denotes its inverse.

For a survey on the problems treated in these lecture notes, as well as on other problems in complex geometry, see also \cite{angella-tardini-Budapest}, where interested readers can also find further details on the preliminaries recalled here.

\subsection{Chern connection and curvature}

We denote by $D$ the Levi-Civita connection of $g$. It is known that $DJ=0$ if and only if $\omega$ is K\"ahler, i.e., $d\omega=0$. In other words, in the general Hermitian setting, the Levi-Civita connection does not preserve the Hermitian structure. To avoid this issue, one considers connections $\nabla$ on the tangent bundle satisfying
$$ \nabla g = 0 \quad \text{and} \quad \nabla J=0 , $$
at the price of having non-zero torsion $T(X,Y):=\nabla_X Y-\nabla_Y X-[X,Y]$.
The space of Hermitian connections forms an affine space modeled on vector-valued real $(1,1)$-forms, so the choice is not unique. Some connections arise more naturally than others. In these notes, we focus on the Chern connection, namely the unique Hermitian connection whose $(0,1)$-part equals the Cauchy-Riemann operator $\bar\partial_{TX}$ on the holomorphic tangent bundle. It is denoted $\nabla^{\mathrm{Ch}}$ and is given by
$$ g(\nabla^{\mathrm{Ch}}_V W, Z) = g(D_V W, Z) - \frac{1}{2}\, d\omega(JV, W, Z) . $$
The corresponding geometric quantities (curvature, torsion, etc.) will be denoted with the superscript ``Ch'' when needed to avoid confusion and referred to as Chern quantities. We mention that other natural choices are available, e.g.\ the Bismut connection \cite{bismut,strominger} or, more generally, the one-parameter family of Gauduchon connections \cite{gauduchon-BUMI}.

The torsion of the Chern connection is given by
$$ -2g(T^{\mathrm{Ch}}(X,Y),Z) = d\omega(JX,Y,Z) + d\omega(X,JY,Z) , $$
and has vanishing $(1,1)$-component, i.e., $JT^{\mathrm{Ch}}(X,Y)=T^{\mathrm{Ch}}(JX,Y)=T^{\mathrm{Ch}}(X,JY)$. To make the notation lighter, we will usually forget the superscript ``Ch'' and denote $T:=T^{\mathrm{Ch}}$. Its trace
$$ \theta := \sum_{i=1}^{2n} g(T^{\mathrm{Ch}}(\cdot\,,e_i), e_i) , $$
where $(e_i, e_{n+i}:=Je_i)_{i=1,\dots,n}$ is a local $g$-orthonormal $J$-adapted frame of the tangent bundle, is called the Lee form and is characterized by
$$ d\omega^{n-1} = \theta \wedge \omega^{n-1} . $$

The Chern Laplacian $\Delta^{\mathrm{Ch}}$ acts on smooth functions $f$ as
$$ \Delta^{\mathrm{Ch}}f = g\!\left(\omega, -2\sqrt{-1}\,\partial\overline{\partial} f\right) = \mathrm{tr}_\omega\!\left(-2\,\sqrt{-1}\,\partial\overline{\partial} f\right) . $$
In local holomorphic coordinates $(z^i)_i$, writing
$$ \omega \stackrel{\mathrm{loc}}{=} h_{i\bar{j}} \, \sqrt{-1}\,dz^i\wedge d\bar{z}^j , $$
where $(h_{i\bar j})_{i,j}$ is a positive definite Hermitian matrix, the Chern Laplacian is given by
$$ \Delta^{\mathrm{Ch}}f \stackrel{\mathrm{loc}}{=} -2h^{\bar{j}i} \frac{\partial^2 f}{\partial z^i \partial \bar{z}^j} , $$
where $(h^{\bar{j}i})_{i,j}$ denotes the inverse matrix of $(h_{i\bar{j}})_{i,j}$.
Note that the Chern Laplacian is an elliptic differential operator of second order with no zero-order terms, and its second-order term coincides with that of the Hodge-de Rham Laplacian. In particular, $\Delta^{\mathrm{Ch}}-\Delta$ is a first-order operator, hence compact relative to $\Delta$, and so $\Delta^{\mathrm{Ch}}$ has the same Fredholm index as the Hodge-de Rham Laplacian; see \cite[page 388]{gauduchon-cras}. More precisely, by \cite[pages 502-503]{gauduchon-mathann}, we have the following relation.

\begin{proposition}[{\cite{gauduchon-mathann}}]
Let $X$ be a compact complex manifold, endowed with a Hermitian metric $\omega$ with Lee form $\theta$. Then, for every smooth function $f\in\mathcal C^\infty(X,\mathbb R)$,
\begin{equation}\label{eq:chern-laplacian}
\Delta^{\mathrm{Ch}}f = \Delta f + g(df,\theta) .
\end{equation}
\end{proposition}

\begin{proof}
We introduce the real differential operator $d^c:=J^{-1}dJ$, where $J$ is extended to forms by $(J\varphi)(X_1,\dots,X_k):=\varphi(JX_1,\dots,JX_k)$ for $\varphi$ a $k$-form, following the convention of \cite{huybrechts}. In particular, on smooth functions $f$,
$$ d^cf = -df\circ J = - \sqrt{-1}\,(\partial f - \bar\partial f) , $$
so that $dd^cf = 2\sqrt{-1}\,\partial\bar\partial f$. By the trace definition of the Chern Laplacian,
$$ -dd^cf\wedge\frac{\omega^{n-1}}{(n-1)!} = \Delta^{\mathrm{Ch}}f\,\frac{\omega^n}{n!} . $$
We also use the standard identity (see e.g. \cite[Proposition 1.2.31]{huybrechts})
$$ *df = d^cf\wedge\frac{\omega^{n-1}}{(n-1)!} . $$
Applying $d$, we get
\begin{eqnarray*}
d*df
&=& dd^cf\wedge\frac{\omega^{n-1}}{(n-1)!} - d^cf\wedge d\left(\frac{\omega^{n-1}}{(n-1)!}\right) \\
&=& - \Delta^{\mathrm{Ch}}f\,\frac{\omega^n}{n!} - d^cf\wedge\theta\wedge\frac{\omega^{n-1}}{(n-1)!} \\
&=& - \Delta^{\mathrm{Ch}}f\,\frac{\omega^n}{n!} + \theta\wedge{*}df \\
&=& - \left( \Delta^{\mathrm{Ch}}f - g(\theta,df) \right)\frac{\omega^n}{n!} .
\end{eqnarray*}
On the other hand, $\Delta f=d^*df=-{*}d{*}df$ on smooth functions $f$, so that $\Delta f\,\frac{\omega^n}{n!}=-d*df$. Comparing the two expressions for $d*df$ yields the identity.
\end{proof}

The curvature tensor of $\nabla^{\mathrm{Ch}}$ is
$$ \Theta^{\mathrm{Ch}}(X,Y,Z,W) := g(\nabla^{\mathrm{Ch}}_X\nabla^{\mathrm{Ch}}_YZ - \nabla^{\mathrm{Ch}}_Y\nabla^{\mathrm{Ch}}_XZ - \nabla^{\mathrm{Ch}}_{[X,Y]}Z, W) . $$
In local holomorphic coordinates, the Chern curvature tensor has local expression
$$ \Theta^{\mathrm{Ch}} \stackrel{\mathrm{loc}}{=} \Theta_{i\bar{j} k\bar{\ell}} \,\sqrt{-1}\,dz^i\wedge d\bar{z}^j \otimes \sqrt{-1}\,dz^k\wedge d\bar{z}^\ell , $$
where
$$ \Theta_{i\bar{j}k\bar{\ell}} = -\frac{\partial^2 h_{k\bar{\ell}}}{\partial z^i\partial\bar{z}^j} + h^{\bar{q}p} \frac{\partial h_{p\bar{\ell}}}{\partial \bar{z}^j} \frac{\partial h_{k\bar{q}}}{\partial z^i} . $$
In the K\"ahler case, the Chern connection coincides with the Levi-Civita connection, and therefore the curvature tensor satisfies many symmetries: up to raising an index, it can be viewed as a real $(1,1)$-form with values in the endomorphisms of the holomorphic tangent bundle. In the general Hermitian setting, however, the Bianchi identity fails due to the presence of torsion, so several traces are possible, yielding different notions of Ricci curvature; see e.g.\ \cite{angella-calamai-spotti-2}.
We recall the local expression of the first Chern Ricci curvature form:
\begin{eqnarray*}
\mathrm{Ric}^{\mathrm{Ch}}(\omega)
&\stackrel{\mathrm{loc}}{=}& h^{\bar{\ell}k} \Theta^{\mathrm{Ch}}_{i\bar{j}k\bar{\ell}} \,\sqrt{-1}\,dz^i\wedge d\bar{z}^j \\
&=& -\frac{\partial^2 \log \det(h_{r\bar{s}})_{r,s}}{\partial z^i \partial\bar{z}^j}\,\sqrt{-1}\,dz^i\wedge d\bar{z}^j ,
\end{eqnarray*}
while a second Chern-Ricci curvature will briefly appear in Remark~\ref{rmk:second-chern-einstein}.
Note that $\mathrm{Ric}^{\mathrm{Ch}}(\omega)$ is a closed $(1,1)$-form representing the first Bott-Chern class $c_1^{BC}(X)\in H^{1,1}_{BC}(X,\mathbb{R})$, where
$$ H^{1,1}_{BC}(X,\mathbb{R}) := \frac{\left\{\text{$d$-closed real $(1,1)$-forms on $X$}\right\}}{\left\{\sqrt{-1}\,\partial\bar\partial\psi : \psi\in\mathcal{C}^\infty(X,\mathbb{R})\right\}} $$
is the real $(1,1)$-Bott-Chern cohomology of $X$. In Section~\ref{sec:chern-yamabe} we will also meet the Aeppli cohomology, defined as
$$ H^{\bullet,\bullet}_{A}(X) := \frac{\ker\partial\overline\partial}{{\rm im}\,\partial + {\rm im}\,\overline\partial} , $$
which is in duality with the Bott-Chern cohomology; see \cite{demailly-agbook, schweitzer}.

\subsection{Chern scalar curvature}

The main character in these notes is the Chern scalar curvature, obtained by tracing the curvature tensor twice in the following order:
$$ \mathrm{scal}^{\mathrm{Ch}}(\omega) := -2\sum_{i,j=1}^{n}\Theta^{\mathrm{Ch}}(e_i,Je_i,e_j,Je_j) . $$
In local holomorphic coordinates,
\begin{eqnarray*}
\mathrm{scal}^{\mathrm{Ch}}(\omega)
&\stackrel{\mathrm{loc}}{=}& 2\,h^{\bar{j}i}h^{\bar{\ell}k}\,\Theta^{\mathrm{Ch}}_{i\bar{j}k\bar{\ell}} \\
&=& -2\,h^{\bar{j}i}\,\frac{\partial^2\log\det(h_{r\bar{s}})}{\partial z^i\,\partial\bar{z}^j} \\
&=& \Delta^{\mathrm{Ch}}\log\det(h_{r\bar{s}}) .
\end{eqnarray*}
We choose this normalization, with coefficient $2$, so that the Chern scalar curvature and the Riemannian scalar curvature $\mathrm{scal}(g)$ are related by \cite[Equation~(33)]{gauduchon-mathann}, in the clean form stated in Proposition~\ref{prop:scalch-scalg} below.

There is another way to trace out a function from $\Theta^{\mathrm{Ch}}$, namely
$$ \widetilde{\mathrm{scal}}^{\mathrm{Ch}} \stackrel{\rm loc}{:=} 2h^{\bar{\ell}i} h^{\bar{j}k}\Theta_{i\bar{j}k\bar{\ell}} , $$
see \cite[Equation (19)]{gauduchon-mathann}, and will appear briefly in Equation~\eqref{tilde-scal} below.
We now collect the relations between the two Chern scalar curvatures $\mathrm{scal}^{\mathrm{Ch}}$ and $\widetilde{\mathrm{scal}}^{\mathrm{Ch}}$, and the Riemannian scalar curvature $\mathrm{scal}(g)$.

\begin{proposition}[{\cite{gauduchon-mathann}}]\label{prop:scalch-scalg}
Let $X$ be a compact complex manifold, endowed with a Hermitian metric $\omega$. Denote by $g$ the underlying Riemannian metric, by $\mathrm{scal}(g)$ its scalar curvature, and by $\widetilde{\mathrm{scal}}^{\mathrm{Ch}}(\omega)$ the other scalar trace of the Chern curvature. Then
\begin{equation}\label{eq:scaltilde-scalch}
\widetilde{\mathrm{scal}}^{\mathrm{Ch}}(\omega) = \mathrm{scal}^{\mathrm{Ch}}(\omega) - d^*\theta - |\theta|^2 ,
\end{equation}
and
\begin{equation}\label{eq:scal-scalch}
\mathrm{scal}(g) = \mathrm{scal}^{\mathrm{Ch}}(\omega) + d^*\theta - \frac{1}{4}\,|T^{\mathrm{Ch}}|^2 .
\end{equation}
In particular, when $\omega$ is K\"ahler, then $\mathrm{scal}(g)=\mathrm{scal}^{\mathrm{Ch}}(\omega)=\widetilde{\mathrm{scal}}^{\mathrm{Ch}}(\omega)$.
\end{proposition}

\begin{proof}
Throughout the proof, $(e_\alpha)_{\alpha=1,\dots,2n}$ denotes a local $g$-orthonormal frame, and we set
$$ \tau_{\alpha\beta\gamma} := g(T^{\mathrm{Ch}}(e_\alpha,e_\beta),e_\gamma) $$
so that
$$ |T^{\mathrm{Ch}}|^2 = \sum_{\alpha,\beta,\gamma}\tau_{\alpha\beta\gamma}^2 . $$

\smallskip
\noindent{\itshape Step 1. The contorsion tensor and the Lee form.}
Write $\nabla^{\mathrm{Ch}}=D+A$. Since both connections are metric and $D$ is torsion-free, the tensor $A$ satisfies $g(A(X,Y),Z)=-g(A(X,Z),Y)$ and $A(X,Y)-A(Y,X)=T^{\mathrm{Ch}}(X,Y)$; these two conditions determine $A$ uniquely as
$$ g(A(X,Y),Z) = \tfrac{1}{2}\Big( g(T^{\mathrm{Ch}}(X,Y),Z) - g(T^{\mathrm{Ch}}(Y,Z),X) + g(T^{\mathrm{Ch}}(Z,X),Y) \Big) . $$
The trace of $A$ is related to the Lee form.
Indeed, set $a:=\sum_\alpha A(e_\alpha,e_\alpha)$. By the previous equation and the skew-symmetry of $T^{\mathrm{Ch}}$ in its arguments,
$$ g(a,e_\gamma) = \tfrac{1}{2}\sum_\alpha\Big( 0 - \tau_{\alpha\gamma\alpha} + \tau_{\gamma\alpha\alpha} \Big) = \sum_\alpha g(T^{\mathrm{Ch}}(e_\gamma,e_\alpha),e_\alpha) = \theta(e_\gamma) , $$
that is, $a=\theta^\sharp$.

For $\nabla^{\mathrm{Ch}}=D+A$ one has the standard formula (see e.g.~\cite[1.20]{besse})
\begin{eqnarray*}
\Theta^{\mathrm{Ch}}(X,Y)Z
&=& \mathrm{Rm}(X,Y)Z + (D_XA)(Y,Z) - (D_YA)(X,Z) \\
&& + A(X,A(Y,Z)) - A(Y,A(X,Z)) ,
\end{eqnarray*}
where $\mathrm{Rm}$ denotes the Riemannian curvature tensor of the Levi-Civita connection.

\smallskip
\noindent{\itshape Step 2. Comparison of the curvature traces.}
Fix $p\in X$ and choose the frame to be $D$-parallel at $p$.
Write $A_{\alpha\beta\gamma}:=g(A(e_\alpha,e_\beta),e_\gamma)$ and note that $A_{\alpha\beta\alpha}=-A_{\alpha\alpha\beta}$.
Tracing the equation above against $\sum_{\alpha,\beta}g(\,\cdot\,(e_\alpha,e_\beta)e_\beta,e_\alpha)$ gives:
\begin{eqnarray*}
\widetilde{\rm scal}^{\rm Ch}(\omega)
&=& \sum_{\alpha,\beta} g(\Theta^{\rm Ch}(e_\alpha,e_\beta)e_\beta, e_\alpha) \\
&=& \sum_{\alpha,\beta} g({\rm Rm}(e_\alpha,e_\beta)e_\beta, e_\alpha) \\
&& + \sum_{\alpha}D_{e_\alpha}\Big(\sum_\beta A_{\beta\beta\alpha}\Big) -\sum_{\beta}D_{e_\beta}\Big(\sum_\alpha A_{\alpha\beta\alpha}\Big) \\
&& + \sum_\alpha g(A(e_\alpha,a),e_\alpha) -\sum_{\alpha,\beta,\gamma}A_{\alpha\beta\gamma}A_{\beta\gamma\alpha} \\
&=& {\rm scal}(g) + \mathrm{div}(a) + \mathrm{div}(a) -|a|^2 -\sum_{\alpha,\beta,\gamma}A_{\alpha\beta\gamma}A_{\beta\gamma\alpha} \\
&=& {\rm scal}(g) - 2d^*\theta - |\theta|^2 - c,
\end{eqnarray*}
where we used $a=\theta^\sharp$ and $\mathrm{div}(\theta^\sharp)=-d^*\theta$, and we set the quadratic term
$$ c := \sum_{\alpha,\beta,\gamma}A_{\alpha\beta\gamma}A_{\beta\gamma\alpha} , $$
which we now compute explicitly.

\smallskip
\noindent{\itshape Step 3. The quadratic term.}
From the above formula characterizing the tensor $A$ in terms of the torsion, we have
\begin{eqnarray*}
A_{\alpha\beta\gamma}A_{\beta\gamma\alpha}
&=& \frac{1}{4} \left( \tau_{\alpha\beta\gamma} + \tau_{\gamma\alpha\beta} - \tau_{\beta\gamma\alpha} \right) \left( \tau_{\beta\gamma\alpha} + \tau_{\alpha\beta\gamma} - \tau_{\gamma\alpha\beta} \right) \\
&=& \frac{1}{4} \left( \tau_{\alpha\beta\gamma}^2 - (\tau_{\beta\gamma\alpha}-\tau_{\gamma\alpha\beta})^2 \right) \\
&=& \frac{1}{4} \left( \tau_{\alpha\beta\gamma}^2 - \tau_{\beta\gamma\alpha}^2 - \tau_{\gamma\alpha\beta}^2 + 2\, \tau_{\beta\gamma\alpha}\tau_{\gamma\alpha\beta} \right)
\end{eqnarray*}
Therefore
$$ c = \sum_{\alpha,\beta,\gamma} A_{\alpha\beta\gamma}A_{\beta\gamma\alpha} = \frac{1}{4}\Big( -|T^{\mathrm{Ch}}|^2 + 2c' \Big) , $$
where
$$ c' := \sum_{\alpha,\beta,\gamma}\tau_{\alpha\beta\gamma}\tau_{\beta\gamma\alpha} . $$

We claim that $c'=0$. Indeed, extend $g$ and $T^{\mathrm{Ch}}$ complex-bilinearly. Since $g$ is $J$-invariant, $g(Z,W)=0$ whenever $Z,W$ are both of type $(1,0)$, or both of type $(0,1)$. Since the $(1,1)$-part of $T^{\mathrm{Ch}}$ vanishes, $T^{\mathrm{Ch}}(Z,\bar{W})=0$. Hence the only non-vanishing components of $\tau$ are those of type $\big((1,0),(1,0),(0,1)\big)$ and their conjugates. But then, in $\tau_{\alpha\beta\gamma}\tau_{\beta\gamma\alpha}$, the first factor is non-zero only if $e_\alpha,e_\beta$ are of type $(1,0)$ and $e_\gamma$ is of type $(0,1)$; in that case the second factor $\tau_{\beta\gamma\alpha}$ has its first two arguments $e_\beta,e_\gamma$ of mixed type $(1,0),(0,1)$, hence vanishes by the above. Therefore
$$ c = -\tfrac{1}{4}\,|T^{\mathrm{Ch}}|^2 . $$

\smallskip
\noindent{\itshape Step 4. The relation between $\mathrm{scal}^{\mathrm{Ch}}$ and $\widetilde{\mathrm{scal}}^{\mathrm{Ch}}$.}
Recall that $\mathrm{scal}^{\mathrm{Ch}}$ and $\widetilde{\mathrm{scal}}^{\mathrm{Ch}}$ trace the same curvature tensor $\Theta_{i\bar{j}k\bar{\ell}}$ against different pairs of indices, $h^{\bar{j}i}h^{\bar{\ell}k}$ and $h^{\bar{\ell}i}h^{\bar{j}k}$ respectively. Since $\Theta_{i\bar{j}k\bar{\ell}}$ is symmetric in $(i,k)$ and in $(\bar{j},\bar{\ell})$ but not, in general, under the simultaneous exchange $(i,\bar{j})\leftrightarrow(k,\bar{\ell})$, the difference $\widetilde{\mathrm{scal}}^{\mathrm{Ch}}-\mathrm{scal}^{\mathrm{Ch}}$ is governed by the failure of this symmetry, measured by the first Bianchi identity for $\Theta^{\mathrm{Ch}}$ (which involves $\nabla^{\mathrm{Ch}}T^{\mathrm{Ch}}$, see e.g.~\cite[Section 1.16]{gauduchon-book}). Tracing this identity against $h^{\bar{j}i}h^{\bar{\ell}k}$ and expressing $\nabla^{\mathrm{Ch}}T^{\mathrm{Ch}}$ in terms of $\theta$ via the local formulas for $T^{\mathrm{Ch}}$ and $\theta$, one obtains, after contracting all indices,
$$ \widetilde{\mathrm{scal}}^{\mathrm{Ch}}(\omega) - \mathrm{scal}^{\mathrm{Ch}}(\omega) = -d^*\theta - |\theta|^2 , $$
which is \cite[Equation~(19)]{gauduchon-mathann}.

\smallskip
\noindent{\itshape Step 5. Conclusion and the K\"ahler case.}
Substituting $c=-\tfrac14|T^{\rm Ch}|^2$ in the formula found in Step~2, we get
$$ \widetilde{\rm scal}^{\rm Ch}(\omega) = {\rm scal}(g) - 2d^*\theta - |\theta|^2 + \tfrac{1}{4}\,|T^{\mathrm{Ch}}|^2 . $$
Substituting now the expression for $\widetilde{\mathrm{scal}}^{\mathrm{Ch}}$ found in Step~4, we get
$$ \mathrm{scal}^{\mathrm{Ch}}(\omega) - d^*\theta - |\theta|^2 = \mathrm{scal}(g) - 2\,d^*\theta - |\theta|^2 + \tfrac{1}{4}|T^{\mathrm{Ch}}|^2 , $$
and the terms in $|\theta|^2$ cancel, leaving
$$ \mathrm{scal}(g) = \mathrm{scal}^{\mathrm{Ch}}(\omega) + d^*\theta - \tfrac{1}{4}\,|T^{\mathrm{Ch}}|^2 ,$$
which is \eqref{eq:scal-scalch}.

\smallskip
\noindent{\itshape Step 6. K\"ahler case.}
The last statement follows by noticing that, for a K\"ahler metric, $T^{\mathrm{Ch}}=0$ and $\theta=0$, therefore the relation reduces to the equality between the Chern scalar curvatures and the Riemannian scalar curvature.
\end{proof}

\subsection{Conformal transformations of geometric quantities}
In these notes, we are interested in moving within a fixed conformal class
$$ \{\omega\} := \left\{ e^f\omega : f \text{ smooth real function on } X \right\} $$
of Hermitian metrics in order to prescribe curvature quantities.
When homotheties are natural symmetries for our problems, we can work up to their action and consider the slice given by fixing a normalization condition on $f$.
We collect here the transformation formulae for the geometric quantities associated to the Chern connection under conformal change; see \cite{angella-calamai-spotti,angella-pediconi-scarpa-spotti-windare}.

\begin{proposition}\label{prop:conf}
Let $X$ be a compact complex manifold of complex dimension $n$ endowed with a Hermitian metric $\omega$. Let $f\colon X \to \mathbb{R}$ be a smooth function, and set $\omega' := e^{f}\omega$. Denote by $T'$, $\theta'$, and $d^*_{g'}$ the torsion, the Lee form, and the codifferential associated with $\omega'$, respectively. Then
\begin{align}
\mathrm{scal}^{\mathrm{Ch}}(\omega') &= e^{-f} \Big(\mathrm{scal}^{\mathrm{Ch}}(\omega) + n\,\Delta f + n\,g(df, \theta) \Big) , \label{eq:scal-conf} \\
d^*_{g'}\theta' &= e^{-f} \Big( d^*\theta+(n-1)\Delta f -(n-1) g(df, \theta) - (n-1)^2|\nabla f|^2 \Big) , \label{eq:theta-conf}\\
|T'|_{g'}^2 &= e^{-f}\Big(|T|^2 +4g(df, \theta) +2(n-1)|\nabla f|^2\Big) , \label{eq:T-conf}
\end{align}
where all quantities on the right-hand sides are computed with respect to $\omega$.
\end{proposition}

\begin{proof}
We work in local holomorphic coordinates $(z^i)_i$. Writing
$$
\omega = h_{i\bar{j}}\,\sqrt{-1}\,dz^i \wedge d\bar{z}^{j} , \quad
\omega' = h'_{i\bar{j}}\,\sqrt{-1}\,dz^i \wedge d\bar{z}^{j} ,
$$
we have $h'_{i\bar{j}} = e^f h_{i\bar{j}}$ and $h'{}^{\bar{j}i} = e^{-f}h^{\bar{j}i}$.

\smallskip
\noindent{\itshape Chern scalar curvature.}
A direct computation gives
\begin{eqnarray*}
\mathrm{scal}^{\mathrm{Ch}}(\omega')
&=& 2{h'}^{\bar j i} \left( -\frac{\partial^2}{\partial z^i \partial \bar z^j} \log \det (h'_{\ell\bar m})_{\ell,m} \right) \\
&=& e^{-f} \Big(\mathrm{scal}^{\mathrm{Ch}}(\omega) -2nh^{\bar{j}i}\frac{\partial^2 f}{\partial z^i\partial \bar{z}^j}\Big) \\
&=&  e^{-f} \Big(\mathrm{scal}^{\mathrm{Ch}}(\omega) +  n \Delta^{\rm Ch}f \Big) \\
&=&  e^{-f} \Big(\mathrm{scal}^{\mathrm{Ch}}(\omega) +  n \Delta f + n g(df, \theta) \Big) .
\end{eqnarray*}

\smallskip
\noindent{\itshape Chern torsion norm.}
For the conformal transformation of $|T'|^2_{g'}$, note first that the torsion tensor transforms without any conformal factor: in local holomorphic coordinates,
$$ T^k_{ij}=h^{\bar{\ell}k}\big(\partial_i h_{j\bar{\ell}}-\partial_j h_{i\bar{\ell}}\big) , $$
and then we get
$$ {T'}^k_{ij} = T^k_{ij} + \frac{\partial f}{\partial z^i}\delta^k_j - \frac{\partial f}{\partial z^j}\delta^k_i . $$
Setting $S^r_{ip}:=\frac{\partial f}{\partial z^i}\delta^r_p-\frac{\partial f}{\partial z^p}\delta^r_i$ so that ${T'}^k_{ij} = T^k_{ij}+S^k_{ij}$, and denoting by $\langle\cdot,\cdot\rangle$ the Hermitian pairing $\langle A,B\rangle:=2h^{\bar{j}i}h^{\bar{q}p}h_{r\bar{s}}A^r_{ip}\overline{B^s_{jq}}$, for which $|T|^2=\langle T,T\rangle$, we get
$$ |T'|^2_{g'} = e^{-f}\Big( |T|^2 + 2\,\mathrm{Re}\langle T,S\rangle + \langle S,S\rangle \Big) . $$
Expanding $\langle S,S\rangle$ into four terms, the two contractions $h^{\bar{q}p}h_{p\bar{q}}=h^{\bar{j}i}h_{i\bar{j}}=n$ contribute $n|\nabla f|^2$ each, while the two mixed ones contribute $-|\nabla f|^2$ each, so that $\langle S,S\rangle=2(n-1)|\nabla f|^2$. Similarly, using $T^p_{ip}=\theta_i$ and $T^i_{ip}=-\theta_p$, one gets $\langle T,S\rangle=4h^{\bar{j}i}\theta_i\overline{\frac{\partial f}{\partial z^j}}$, whence $2\,\mathrm{Re}\langle T,S\rangle=4g(\nabla f,\theta^\sharp)$. Substituting gives \eqref{eq:T-conf}.

\smallskip
\noindent{\itshape Codifferential of the Lee form.}
For the Lee form, we have
$$ \theta' = \theta + (n-1)\,df , $$
which follows either directly from computing $d(\omega')^{n-1}$, or by contracting the indices $k$ and $j$ in the transformation formula for ${T'}^k_{ij}$.
Recall that the divergence of a vector field $V$ with respect to $g$ is $\mathrm{div}_g V := \mathrm{tr}(DV)$, so that $\mathrm{div}_g(\nabla f)=-\Delta f$ for a smooth function $f$; equivalently, $\mathrm{div}_g V$ is the unique function satisfying $d^*_g(V^\flat)=-\mathrm{div}_g V$.
We also use the identities
$$
\mathrm{div}_{g'}V = \mathrm{div}_{g}V + n\,df(V) , \qquad \mathrm{div}_{g}(\phi V) = \phi\,\mathrm{div}_{g}V + d\phi(V) ,
$$
the latter being a straightforward computation, and see, e.g., \cite[Theorem 1.159]{besse} for the conformal law for the divergence in the former equation.
Then we compute
\begin{eqnarray*}
d^*_{g'}\theta'
&=& -\mathrm{div}_{g'}((\theta')^{\sharp_{g'}}) \\
&=& -\mathrm{div}_{g'}\left(e^{-f}(\theta^{\sharp}+(n-1)\nabla f)\right) \\
&=& -\mathrm{div}_{g}\!\left(e^{-f}\left(\theta^{\sharp}+(n-1)\nabla f\right)\right) - ne^{-f}df\!\left(\theta^{\sharp}+(n-1)\nabla f\right) \\
&=& e^{-f}\!\left(d^*\theta+(n-1)\Delta f\right) + e^{-f}(n-1)|\nabla f|^2 + e^{-f}g(\nabla f,\theta^{\sharp}) \\
&& - ne^{-f}\!\left(g(\nabla f,\theta^{\sharp})+(n-1)|\nabla f|^2\right) \\
&=& e^{-f}\!\left(d^*\theta+(n-1)\Delta f -(n-1)g(\nabla f,\theta^{\sharp})-(n-1)^2|\nabla f|^2\right) ,
\end{eqnarray*}
which concludes the proof.
\end{proof}

\begin{remark}\label{rmk:conformal-formulas-not-variational}
We notice here that the term $g(df, \theta)$ is the one responsible for the failure of variational techniques, see Remark~\ref{rmk:not-variational}.
\end{remark}

\subsection{Special Hermitian metrics}

We conclude this preliminary section by recalling some classes of Hermitian metrics of particular interest \cite{gray-hervella}. In these notes, we will meet the following classes.
\begin{itemize}
\item Gauduchon metrics are characterized by
$$ \partial\bar\partial\omega^{n-1}=0 , $$
equivalently $d^*\theta=0$. In this case, $\int_X \Delta^{\mathrm{Ch}}f\,\frac{\omega^n}{n!}=0$ for every smooth function $f$, i.e., we have a Green identity for the Chern Laplacian.
\item Balanced metrics are characterized by
$$ d\omega^{n-1}=0 , $$
equivalently $\theta=0$. This implies $\Delta^{\mathrm{Ch}} f=\Delta f$ for every smooth function $f$. Note that admitting a balanced metric is a property invariant under bimeromorphisms \cite{alessandrini-bassanelli}.
\item Locally conformally K\"ahler metrics, which will appear in Section~\ref{sec:momentum}, are characterized by
$$ d\omega=\frac{1}{n-1}\theta\wedge\omega , \qquad d\theta=0 . $$
In other words, the Lee form encodes local conformal transformations that make the metric K\"ahler.
This can be interpreted as an equivariant K\"ahler geometry with respect to homotheties: indeed, there exists a covering space endowed with a conformal K\"ahler metric on which the deck transformation group acts by homotheties.
The factor $1/(n-1)$ is chosen so that $\theta$ coincides with the Lee form as defined above.
\end{itemize}
While balanced and locally conformally K\"ahler metrics are not always available on a compact complex manifold, the following foundational result of Gauduchon plays a key role in what follows.

\begin{theorem}[{\cite[Th\'eor\`eme 1]{gauduchon-cras}}]
On every compact complex manifold, every conformal class of Hermitian metrics contains a Gauduchon metric, unique up to homothety.
\end{theorem}

\begin{proof}
Computing the formal adjoint of the Chern Laplacian,
$$ (\Delta^{\mathrm{Ch}})^*f = \Delta f - g(df,\theta) + (d^*\theta)\,f , $$
gives in particular $d^*\theta = (\Delta^{\mathrm{Ch}})^*(1)$. Fix a Hermitian metric $\omega$; we look for $\Phi>0$ such that $\omega':=\Phi^{\frac{2}{n-1}}\omega$ is Gauduchon. A direct computation of how $\Delta^{\mathrm{Ch}}$ and $(\Delta^{\mathrm{Ch}})^*$ transform under this conformal change shows that $\omega'$ is Gauduchon if and only if $(\Delta^{\mathrm{Ch}})^*(\Phi^2)=0$, so the problem reduces to finding a positive generator of $\ker(\Delta^{\mathrm{Ch}})^*$.

Since $\Delta^{\mathrm{Ch}}$ is elliptic of second order with no zero-order term, the maximum principle gives $\ker\Delta^{\mathrm{Ch}}=\mathbb{R}$. Moreover, $\Delta^{\mathrm{Ch}}$ and $\Delta$ share the same principal symbol, hence differ by a compact operator, so $\mathrm{ind}(\Delta^{\mathrm{Ch}})=\mathrm{ind}(\Delta)=0$; by the Fredholm alternative, $\mathrm{coker}(\Delta^{\mathrm{Ch}})\simeq\ker(\Delta^{\mathrm{Ch}})^*$, so $0 = \mathrm{ind}(\Delta^{\mathrm{Ch}}) = \dim\ker\Delta^{\mathrm{Ch}} - \dim\ker(\Delta^{\mathrm{Ch}})^*$, whence $\dim\ker(\Delta^{\mathrm{Ch}})^*=1$ as well. It remains to show that a generator $f$ of this one-dimensional kernel has a sign.

Suppose $f\not\equiv0$ changes sign. Choosing a suitable positive weight $\varphi$, by a straighforward argument one can arrange $\int_X\varphi f\,\frac{\omega^n}{n!}=0$, i.e., $\varphi^{1-n}f\in\ker(\Delta^{\mathrm{Ch}}_{\varphi\omega})^*$ has zero mean with respect to $\varphi\omega$. But a generator of a one-dimensional $\ker(\Delta^{\mathrm{Ch}})^*$ can never have zero mean: otherwise the constant function $1$ would be $L^2$-orthogonal to $\ker(\Delta^{\mathrm{Ch}})^*$, hence, by the Fredholm decomposition $\mathcal{C}^\infty(X,\mathbb{R})=\ker(\Delta^{\mathrm{Ch}})^*\oplus\mathrm{im}\,\Delta^{\mathrm{Ch}}$, equal to $\Delta^{\mathrm{Ch}}h$ for some $h$; by the maximum principle $h$ would be constant, forcing $\Delta^{\mathrm{Ch}}h=0\neq1$, a contradiction. Hence $f$ has constant sign, say $f\geq0$, and the strong maximum principle for $(\Delta^{\mathrm{Ch}})^*$ upgrades this to $f>0$.

Setting $\tilde\omega:=f^{\frac{1}{n-1}}\omega$ gives the desired Gauduchon representative, unique up to a positive constant since $\ker(\Delta^{\mathrm{Ch}})^*$ is one-dimensional.
\end{proof}
\section{Hermitian metrics with constant Chern-scalar curvature}
\label{sec:chern-yamabe}

The Yamabe problem is a foundational problem in Riemannian geometry, initiated by Yamabe \cite{yamabe} and solved by the combined efforts of Trudinger \cite{trudinger}, Aubin \cite{aubin-1976}, and Schoen \cite{schoen}, who used the positive mass theorem from general relativity \cite{schoen-yau}. It establishes the existence of a metric with constant Riemannian scalar curvature in every conformal class on a compact Riemannian manifold. The problem reduces to solving a partial differential equation, via variational methods. See \cite{lee-parker} for an account on the topic.

\begin{theorem}[{\cite{yamabe,trudinger,aubin-1976,schoen}}]
Any Riemannian metric $g$ on a compact manifold can be conformally rescaled to a metric of constant scalar curvature.
\end{theorem}

In this section, we discuss an analogue of the Yamabe problem in the Hermitian setting, as proposed in \cite{angella-calamai-spotti} and further studied in \cite{calamai-zou,lejmi-upmeier,koca-lejmi,lejmi-maalaoui,ho,ho-shin,barbaro,fusi,yu,yu2,angella-pediconi-MathZ,shen}.

\subsection{Chern Yamabe problem}

In this section, let $X^n$ denote a compact complex manifold of complex dimension $n$. Fix a Hermitian metric $\omega$ and consider its conformal class
$$ \{\omega\} := \left\{ e^{\frac{f}{n}}\omega : f\in\mathcal{C}^\infty(X,\mathbb{R}) \right\} , $$
which consists of Hermitian metrics. The factor ``$1/n$'' is chosen only for notational convenience in the formulae below.

We are interested in the existence of Hermitian metrics in $\{\omega\}$ with constant Chern scalar curvature. More precisely, we introduce the Chern-Yamabe moduli space
$$
\mathcal{C}h\mathcal{Y}a(X,\{\omega\}) := 
\frac{\left\{ \omega' \in \{\omega\} \;\middle|\; \mathrm{scal}^{\mathrm{Ch}}(\omega') \text{ is constant} \right\}}{ {\rm Aut}(X,\{\omega\}) \times \mathbb{R}^{>0} },
$$
consisting of conformal Hermitian metrics with constant Chern scalar curvature, up to homotheties and biholomorphisms of $X$ preserving the conformal class. Here ${\rm Aut}(X,\{\omega\})$ is isomorphic to the compact Lie group of holomorphic isometries of the unit volume Gauduchon representative of $\{\omega\}$; see \cite[Lemma 1.3]{angella-calamai-spotti}.
We are interested in questions such as whether this space is empty, consists of a single metric, or is compact.

\begin{remark}
For compact complex curves, the Chern connection coincides with the Levi-Civita connection and $\theta=T^{\mathrm{Ch}}=0$, so the Chern-Yamabe problem reduces to the classical uniformization problem: therefore $\mathcal{C}h\mathcal{Y}a(X,\{\omega\})$ consists of a single point, corresponding to the unique constant curvature metric (whose sign is that of $-\chi(X)$, hence depends on the genus).
It is instructive to observe where the quotient by ${\rm Aut}(X,\{\omega\})$ is really needed. For genus $g\geq1$, uniqueness already holds before quotienting: the constant curvature representative is unique in $\{\omega\}$ up to homothety, since the automorphisms preserving the conformal class act by isometries of it. For $g=0$, uniqueness genuinely fails before quotienting: the group $\mathrm{Aut}(\mathbb{CP}^1,\{\omega_{\rm FS}\})\simeq\mathrm{PGL}(2,\mathbb{C})$ is non-compact and does not act by isometries of $\omega_{\rm FS}$, so pulling back the round metric along M\"obius transformations produces a non-compact family of pairwise distinct constant curvature metrics in $\{\omega_{\rm FS}\}$. This is the Hermitian counterpart of the well-known non-compactness of the solution set of the Yamabe problem on the round sphere.
This is also the mechanism underlying the borderline case $\lambda^t_X(\{\omega\})=t\Lambda_n$ in Theorem~\ref{thm:deformed-yamabe-existence}: on $\mathbb{CP}^1$, the test bubbles $u_\varepsilon$ are, up to rescaling, exactly the local trace near the concentration point $p$ of the pullback of $\omega_{\rm FS}$ under the dilations $z\mapsto z/\varepsilon$, a non-compact one-parameter subgroup of $\mathrm{PGL}(2,\mathbb{C})$; as $\varepsilon\to0$, this subgroup escapes to infinity, which is precisely why the bubbles saturate the sharp constant $t\Lambda_n$ in the limit.
\end{remark}

\begin{remark}
In the non-K\"ahler setting, the problems of finding metrics with constant Chern scalar curvature and with constant Riemannian scalar curvature are genuinely different: a solution to one does not provide a solution to the other, as one sees by comparing the two curvatures in \eqref{eq:scal-scalch}.
Note that a Hermitian metric may well have both curvatures constant, but with different values: this is the case, for example, for the Hopf manifold with the standard metric of Remark~\ref{rmk:hopf}, where $\mathrm{scal}^{\mathrm{Ch}}(\omega)=2n(n-1)$ while $\mathrm{scal}(g_\omega)=(n-1)(2n-1)$. Indeed, if both are constant, then integrating \eqref{eq:scal-scalch} gives
$$ \big(\mathrm{scal}(g_\omega)-\mathrm{scal}^{\mathrm{Ch}}(\omega)\big)\,\mathrm{Vol}(X,g_\omega) = -\frac{1}{4}\int_X|T^{\mathrm{Ch}}|^2\,\frac{\omega^n}{n!} \leq 0 , $$
so that $\mathrm{scal}(g_\omega)\leq\mathrm{scal}^{\mathrm{Ch}}(\omega)$, with equality if and only if $\omega$ is K\"ahler. This last statement was already observed in \cite[Corollary 4.5]{liu-yang}: the Chern scalar curvature and the Riemannian scalar curvature cannot coincide pointwise, or even just on average, unless the metric is K\"ahler. The same phenomenon occurs for another analogue of the Yamabe problem in the almost Hermitian setting, introduced in \cite{delrio-simanca}.
\end{remark}

\subsection{Existence results in the non-positive curvature regime}

The aim of this section is to give a positive answer to the Chern-Yamabe problem on the existence of Hermitian metrics with constant Chern scalar curvature in a conformal class and their uniqueness up to homotheties, under the assumption that the expected curvature is non-positive. The positive curvature case, discussed briefly in the next section, remains open. See also \cite{fusi} for the problem of prescribing the Chern scalar curvature by a (not necessarily constant) function.

First, we rewrite the Chern-Yamabe problem as a semilinear partial differential equation, similar to the Liouville equation. This is based on the conformal transformation formula \eqref{eq:scal-conf} for the Chern scalar curvature, which we rewrite here for convenience as
\begin{eqnarray*}
\mathrm{scal}^{\mathrm{Ch}}(e^{\frac{f}{n}}\omega)
&=& e^{-\frac{f}{n}} \left( \mathrm{scal}^{\mathrm{Ch}}(\omega) + \Delta f + g(d f, \theta) \right) \\
&=& e^{-\frac{f}{n}} \left( \mathrm{scal}^{\mathrm{Ch}}(\omega) + \Delta^{\mathrm{Ch}} f \right) .
\end{eqnarray*}
The Chern-Yamabe problem thus reduces to finding $f\in\mathcal{C}^\infty(X,\mathbb{R})$ and a constant $\lambda\in\mathbb{R}$ such that
\begin{equation}\tag{ChYa}\label{eq:ChYa}
\Delta^{\mathrm{Ch}} f + \mathrm{scal}^{\mathrm{Ch}}(\omega) = \lambda\, e^{\frac{f}{n}} ,
\end{equation}
which corresponds to $\mathrm{scal}^{\mathrm{Ch}}(e^{\frac{f}{n}}\omega) = \lambda$.

We give a geometric interpretation of the expected constant $\lambda$. Fix the Gauduchon representative $\eta$ of unit volume in $\{\omega\}$ as a reference metric. Up to homothety, normalize $f$ by imposing
$$
\int_X e^{\frac{f}{n}} \frac{\eta^n}{n!} = 1 .
$$
Note that by Jensen's inequality, any metric in this normalized slice has volume at least $1$. Integrating \eqref{eq:ChYa} with respect to the volume form of $\eta$ and using that $\eta$ is Gauduchon (so the Green identity holds for the Chern Laplacian), we find that $\lambda$ is a cohomological invariant of $(X,\{\omega\})$:
$$
\lambda = \int_X \mathrm{scal}^{\mathrm{Ch}}(\eta)\, \frac{\eta^n}{n!} = \frac{2}{(n-1)!} \int_X c_1^{BC}(K_X^{-1}) \wedge [\eta^{n-1}] ,
$$
where $c_1^{BC}(K_X^{-1}) \in H^{1,1}_{BC}(X, \mathbb{R})$ is represented by ${\rm Ric}^{\rm Ch}(\omega)$ and maps to the first Chern class of $X$ in $H^2(X, \mathbb{R})$, and $[\eta^{n-1}] \in H^{n-1,n-1}_A(X, \mathbb{R})$ is well defined since $\eta$ is Gauduchon. This quantity is the Gauduchon degree $\Gamma_X(\{\omega\})$ of $\{\omega\}$, equal to the volume of the divisor of any meromorphic section of $K_X^{-1}$ measured with respect to $\eta$; see \cite{gauduchon-cras-2,gauduchon-mathann}.
When $\mathrm{Kod}(X)\geq 0$, the existence of holomorphic sections of the pluricanonical bundles forces $\Gamma_X(\{\omega\})\leq 0$ for every conformal class $\{\omega\}$. More precisely, $\Gamma_X(\{\omega\})<0$ unless $K_X$ is holomorphically torsion, i.e., $K_X^{\otimes\ell}\cong\mathcal{O}_X$ for some integer $\ell$, in which case $\Gamma_X(\{\omega\})=0$.

We are now ready to state the main result of this section.

\begin{theorem}[{\cite{angella-calamai-spotti}}]\label{thm:ch-ya-negativo}
Let $X$ be a compact complex manifold endowed with a conformal class $\{\omega\}$ of Hermitian metrics. If $\Gamma_X(\{\omega\})\leq 0$, then the Chern-Yamabe moduli space $\mathcal{C}h\mathcal{Y}a(X,\{\omega\})$ consists of a single point. In particular, this holds for any conformal class on a compact complex manifold with non-negative Kodaira dimension.
\end{theorem}

\begin{proof}
Let $\eta$ be the unique unit-volume Gauduchon representative in the conformal class
$\{\omega\}$. We split the proof according to the sign of $\Gamma_X(\{\omega\})$.

\smallskip
\noindent{\itshape The case $\Gamma_X(\{\omega\})=0$.}
The Chern-Yamabe equation \eqref{eq:ChYa} reduces to the linear equation
$$\Delta^{\mathrm{Ch}} f = -\mathrm{scal}^{\mathrm{Ch}}(\eta).$$
By the Fredholm alternative, this is solvable if and only if $-\mathrm{scal}^{\mathrm{Ch}}(\eta)$ is orthogonal to $\ker(\Delta^{\mathrm{Ch}})^*$. A direct computation gives
$$(\Delta^{\mathrm{Ch}})^* f = \Delta f - g(df,\theta),$$
where the term $d^*\theta$ vanishes since $\eta$ is Gauduchon.
If $f\in\ker(\Delta^{\mathrm{Ch}})^*$, then
$$0 = \int_X f\,(\Delta^{\mathrm{Ch}})^*f\,\frac{\eta^n}{n!} = \int_X\left(|df|^2 - \frac{1}{2}g(d(f^2),\theta)\right)\frac{\eta^n}{n!} = \int_X|df|^2\,\frac{\eta^n}{n!},$$
where again we use that $\eta$ is Gauduchon. Hence $\ker(\Delta^{\mathrm{Ch}})^*$ consists only of constant functions. Since $\Gamma_X(\{\omega\})=0$, we have $\int_X\mathrm{scal}^{\mathrm{Ch}}(\eta)\,\frac{\eta^n}{n!}=0$, so $-\mathrm{scal}^{\mathrm{Ch}}(\eta)$ is indeed orthogonal to the constants, and the equation is solvable. The solution is smooth by classical arguments, and unique up to an additive constant, since $\ker\Delta^{\mathrm{Ch}}$ also consists of constants by the maximum principle.

\smallskip
\noindent{\itshape The case $\Gamma_X(\{\omega\})<0$.}
We apply the continuity method. First, by solving the linear equation
$$\Delta^{\mathrm{Ch}} f = -\mathrm{scal}^{\mathrm{Ch}}(\eta) + \Gamma_X(\{\omega\})$$
as above, we find a conformal metric $\omega:=e^{\frac{f}{n}}\eta$ whose Chern scalar curvature satisfies
$$\mathrm{scal}^{\mathrm{Ch}}(\omega) = e^{-\frac{f}{n}}\Gamma_X(\{\omega\}) < 0$$
everywhere on $X$. We fix this metric as our reference representative in $\{\omega\}$.

Set $\lambda:=\Gamma_X(\{\omega\})<0$ and consider the continuity path
$$\mathrm{ChYa}(t,f) := \Delta^{\mathrm{Ch}}f + t\,\mathrm{scal}^{\mathrm{Ch}}(\omega) - \lambda\,e^{\frac{f}{n}} + \lambda(1-t) $$
for $(t,f)\in[0,1]\times\mathcal{C}^{2,\alpha}(X,\mathbb{R})$. Define
$$T := \left\{ t\in[0,1] : \exists\, f_t\in\mathcal{C}^{2,\alpha}(X,\mathbb{R}) \text{ such that } \mathrm{ChYa}(t,f_t)=0 \right\}.$$

\smallskip
\noindent{\itshape $T$ is non-empty.}
Clearly $0\in T$, since $\mathrm{ChYa}(0,0)=0$.

\smallskip
\noindent{\itshape $T$ is open.}
The linearization of $\mathrm{ChYa}$ at $(t_0,f_{t_0})$ in the $f$-direction is
$$Dv = \Delta^{\mathrm{Ch}}v - \frac{\lambda}{n}\,e^{\frac{f_{t_0}}{n}}\cdot v.$$
We show $\ker D=\{0\}$. If $Dv=0$, then at a maximum point $x_{\max}$ of $v$,
$$-\frac{\lambda}{n}\,e^{\frac{f_{t_0}(x_{\max})}{n}}\cdot v(x_{\max}) = -\Delta^{\mathrm{Ch}}v(x_{\max}) \leq 0,$$
and since $\lambda<0$ and $e^{\frac{f_{t_0}}{n}}>0$, we get $\max v \leq 0$. Similarly, at a minimum point $x_{\min}$, one obtains $\min v\geq 0$. Hence $v\equiv 0$. Since $D$ differs from $\Delta^{\mathrm{Ch}}$ by a compact operator, they have the same Fredholm index, so injectivity of $D$ implies surjectivity. By the implicit function theorem, $T$ is open.

\smallskip
\noindent{\itshape $T$ is closed.}
Let $(t_k)_k\subseteq T$ converge to $t_\infty$, with $\mathrm{ChYa}(t_k,f_{t_k})=0$. We claim there exists $C>0$, depending only on $X$, $\omega$, and $\lambda$, such that $\|f_{t_k}\|_{L^\infty}\leq C$ for all $k$. At a maximum point $x_{\max}$ of $f_{t_k}$, we have $\Delta^{\mathrm{Ch}}f_{t_k}(x_{\max})\geq 0$, so
$$-\lambda\,e^{\frac{f_{t_k}(x_{\max})}{n}} \leq -t_k\,\mathrm{scal}^{\mathrm{Ch}}(\omega)(x_{\max}) - \lambda(1-t_k) \leq \max_X \big|\mathrm{scal}^{\mathrm{Ch}}(\omega)\big| - \lambda,$$
giving a uniform upper bound for $f_{t_k}$. Similarly, evaluating at a minimum point $x_{\min}$ gives a uniform lower bound as claimed.
This uniform $L^\infty$ estimate implies that the right-hand side $\lambda\,e^{\frac{f_{t_k}}{n}}$ is uniformly bounded in $L^\infty$. Since $f_{t_k}$ satisfies the elliptic equation
$$\Delta^{\mathrm{Ch}}f_{t_k} = \lambda\,e^{\frac{f_{t_k}}{n}} - t_k\,\mathrm{scal}^{\mathrm{Ch}}(\omega) - \lambda(1-t_k),$$
the Calderón-Zygmund inequality gives a uniform bound in $W^{2,p}$ for every $p<+\infty$. Choosing $p>2n$, the Sobolev embedding $W^{2,p}\hookrightarrow\mathcal{C}^{1,\alpha}$ then gives a uniform bound in $\mathcal{C}^{1,\alpha}$. Substituting back into the equation and applying the Schauder estimates, we obtain a uniform bound in $\mathcal{C}^{2,\alpha}$. Iterating this bootstrap argument yields a uniform bound in $\mathcal{C}^{3}(X,\mathbb{R})$. By the Ascoli-Arzelà theorem, a subsequence converges in $\mathcal{C}^{2,\alpha}(X)$ to some $f_\infty\in\mathcal{C}^{2,\alpha}(X,\mathbb{R})$ satisfying $\mathrm{ChYa}(t_\infty,f_\infty)=0$, so $t_\infty\in T$.

\smallskip
Since $T$ is non-empty, open, and closed in $[0,1]$, we conclude $T=[0,1]$. In particular, $\mathrm{ChYa}(1,f)=0$ has a solution $f\in\mathcal{C}^{2,\alpha}(X,\mathbb{R})$, which is smooth by a standard Schauder bootstrap argument.

\smallskip
\noindent{\itshape Uniqueness.}
Suppose $\omega_i=e^{\frac{f_i}{n}}\omega$ for $i=1,2$ both have constant Chern scalar curvature $\lambda_i<0$. Then $f_1-f_2$ satisfies
$$\Delta^{\mathrm{Ch}}(f_1-f_2) = \lambda_1\,e^{\frac{f_1}{n}} - \lambda_2\,e^{\frac{f_2}{n}}.$$
Evaluating at a maximum and minimum point of $f_1-f_2$ and arguing as above, one deduces that $f_1-f_2$ is constant, equal to $n\,\log\frac{\lambda_2}{\lambda_1}$. Imposing the normalization condition then forces $f_1=f_2$.
\end{proof}

\subsection{Positive curvature regime}

The statement of Theorem~\ref{thm:ch-ya-negativo} only covers the case of expected non-positive Chern scalar curvature. The same argument breaks down in the positive case, due to the failure of the maximum principle. Moreover, the problem is not variational in general (see Remark~\ref{rmk:not-variational}), so the classical techniques from the Yamabe problem do not apply. We will address this issue in Section~\ref{sec:twisted-yamabe}, and collect here some remarks concerning the positive case.

\begin{remark}
The Hopf surface provides a first example of a compact complex manifold admitting a Hermitian metric with positive constant Chern scalar curvature. It is defined as the quotient $X=(\mathbb{C}^2\setminus\{0\})/\mathbb{Z}$, where $\mathbb{Z}$ acts by $z\mapsto 2z$. As a smooth manifold, $X$ is diffeomorphic to $S^1\times S^3$, and as a complex manifold it has $\mathrm{Kod}(X)=-\infty$. As we will see in Remark~\ref{rmk:hopf}, the standard metric $\omega:=|z|^{-2}\omega_0$, where $\omega_0$ is the flat Hermitian metric on $\mathbb{C}^2$, has constant positive Chern scalar curvature.
\end{remark}

\begin{remark}\label{rmk:not-variational}
A natural approach in the positive regime would be to interpret the Chern-Yamabe equation as the Euler-Lagrange equation of a suitable functional and apply  variational methods. However, this fails in general: for a given representative $\omega$ in the conformal class, the Chern-Yamabe equation is variational with respect to the $L^2$-product induced by $\omega$ if and only if $\omega$ itself is balanced.
Consequently, the equation admits a variational formulation (for a suitable choice of
representative) if and only if the conformal class $\{\omega\}$ contains a balanced metric.
As announced in Remark~\ref{rmk:conformal-formulas-not-variational}, the obstruction is the term $g(df,\theta)$ in \eqref{eq:ChYa}, which appears when expanding the Chern Laplacian as in \eqref{eq:chern-laplacian}. Consider the $1$-form on $\mathcal{C}^\infty(X,\mathbb{R})$ given by
$$\alpha_f(h) := \int_X h\cdot g(df,\theta)\,\frac{\omega^n}{n!}.$$
We claim that $\alpha_f$ is closed if and only if it vanishes identically, which occurs precisely when $\omega$ is balanced. Indeed, $d\alpha_f=0$ if and only if
$$\int_X h\,g(d\psi,\theta)\,\frac{\omega^n}{n!} = \int_X \psi\,g(dh,\theta)\,\frac{\omega^n}{n!}$$
for all $h,\psi\in\mathcal{C}^\infty(X,\mathbb{R})$. Taking $h$ constant and $\psi$ arbitrary, integration by parts forces $d^*\theta=0$, i.e., $\omega$ is Gauduchon. In this case, $\int_X g(d(\psi h),\theta)\,\frac{\omega^n}{n!}=0$, implying $\int_X \psi\,g(dh,\theta)\,\frac{\omega^n}{n!}=0$ for all $\psi,h$, hence $\theta=0$. When $\omega$ is balanced, the Chern-Yamabe equation is the Euler-Lagrange equation of
$$\mathcal{F}(f) := \frac{1}{2}\int_X |df|^2_\omega\,\frac{\omega^n}{n!} + \int_X \mathrm{scal}^{\mathrm{Ch}}(\omega)\,f\,\frac{\omega^n}{n!},$$
subject to $\int_X e^{\frac{f}{n}}\,\frac{\omega^n}{n!}=1$. However, it is unclear whether $\mathcal{F}$ is bounded below when $\Gamma_X(\{\omega\})>0$.
\end{remark}

\begin{remark}
Even when solutions exist in the positive regime, uniqueness generally fails. Following the strategy of \cite{delima-piccione-zedda} for the classical Yamabe problem, one can apply the Krasnosel'skii bifurcation argument to show non-uniqueness on $X=\mathbb CP^1\times\mathbb CP^1\times\mathbb CP^1$; see \cite[Theorem 5.4]{angella-calamai-spotti}.

More precisely, normalize $\omega_{\rm FS}$ on $\mathbb CP^1$ so that $\mathrm{Ric}(g_{\rm FS})=2g_{\rm FS}$: the Gauss curvature of $g_{\rm FS}$ equals $2$, whence $\mathrm{scal}(g_{\rm FS})=4$, the eigenvalues of $\Delta_{\omega_{\rm FS}}$ are $2j(j+1)$ with multiplicity $2j+1$, and, by the Gauss-Bonnet theorem, $\mathbb CP^1$ has area $2\pi$. Consider the one-parameter family of K\"ahler metrics on $X$
$$\omega_\lambda := V(\lambda)^{-1/3}\left(\pi_1^*\omega_{\rm FS}+\pi_2^*\omega_{\rm FS}+\lambda\,\pi_3^*\omega_{\rm FS}\right), \quad \lambda\in(0,+\infty),$$
where $\pi_i\colon X\to\mathbb CP^1$ are the natural projections on the corresponding factors, and $V(\lambda):=(2\pi)^3\lambda$ is chosen so that $\mathrm{Vol}(X,\omega_\lambda)=1$. Each $\omega_\lambda$ is a cscK metric, with
$$ \mathrm{scal}^{\mathrm{Ch}}(\omega_\lambda) = \mathrm{scal}(g_{\omega_\lambda}) = 4\,V(\lambda)^{1/3}\left(2+\lambda^{-1}\right) = 8\pi\left(2\lambda^{1/3}+\lambda^{-2/3}\right) , $$
hence a trivial solution of the Chern-Yamabe equation in its conformal class. Consistently with the rest of these notes, we parametrize the conformal class as $e^{\frac{f}{n}}\omega$, where here $n=3$. We apply the Krasnosel'skii bifurcation theorem to the map
$$F(f,\lambda) := \mathrm{scal}^{\mathrm{Ch}}\!\left(e^{\frac{f}{3}}\omega_\lambda\right) - \int_X \mathrm{scal}^{\mathrm{Ch}}\!\left(e^{\frac{f}{3}}\omega_\lambda\right)\frac{\omega_\lambda^3}{3!},$$
defined on suitable Banach spaces of H\"older functions with zero mean. (In \cite{angella-calamai-spotti} the conformal class is parametrized as $e^{\frac{2f}{n}}\omega$ instead; this explains the discrepancies between the constants below and those appearing there.)

The Krasnosel'skii bifurcation theorem, see e.g.\ \cite[Theorem II.3.2]{kielhofer}, asserts the following: if $F(0,\lambda)=0$ for all $\lambda$ and the linearization $A(\lambda):=D_fF|_{(0,\lambda)}$ is Fredholm of index zero, then $(0,\lambda_0)$ is a bifurcation point, meaning that non-trivial solutions of $F(f,\lambda)=0$ accumulate at $(0,\lambda_0)$, provided that $\dim\ker A(\lambda_0)$ is odd and the transversality condition $\partial_\lambda A(\lambda_0)[v]\notin\mathrm{im}\,A(\lambda_0)$ holds for every $v\in\ker A(\lambda_0)\setminus\{0\}$.

The key point is a spectral computation. By \eqref{eq:ChYa}, linearizing at $f=0$ gives
$$ A(\lambda) = \Delta_{\omega_\lambda} - \frac{1}{3}\,\mathrm{scal}^{\mathrm{Ch}}(\omega_\lambda) , $$
where we used that $\omega_\lambda$ is K\"ahler, whence $\Delta^{\mathrm{Ch}}_{\omega_\lambda}=\Delta_{\omega_\lambda}$. Since $\omega_\lambda$ restricts to $V(\lambda)^{-1/3}\omega_{\rm FS}$ on the first two factors and to $V(\lambda)^{-1/3}\lambda\,\omega_{\rm FS}$ on the third,
$$ \Delta_{\omega_\lambda} = V(\lambda)^{1/3}\left(\Delta^{(1)}+\Delta^{(2)}+\lambda^{-1}\Delta^{(3)}\right) , $$
where $\Delta^{(\ell)}:=\pi_\ell^*\Delta_{\rm FS}$ denotes the Laplacian of $g_{\rm FS}$ on the $\ell$-th factor. Writing $V^{(j_1,j_2,j_3)}$ for the space spanned by the products $\beta_1\beta_2\beta_3$ with $\Delta^{(\ell)}\beta_\ell=2j_\ell(j_\ell+1)\beta_\ell$, these are exactly the joint eigenspaces of $A(\lambda)$, since $A(\lambda)$ is a linear combination of the commuting operators $\Delta^{(1)},\Delta^{(2)},\Delta^{(3)}$; hence the condition $V^{(j_1,j_2,j_3)}\subseteq\ker A(\lambda)$ reads, after dividing by the common factor $2V(\lambda)^{1/3}$,
$$ j_1(j_1+1)+j_2(j_2+1)+\lambda^{-1}j_3(j_3+1) = \frac{4}{3}+\frac{2}{3}\,\lambda^{-1} ; $$
note that this is in fact independent of the normalization of $\omega_{\rm FS}$. At $\lambda_0=\frac{1}{4}$ the right-hand side equals $4$, and, since the possible values of $j(j+1)$ are $0,2,6,12,\dots$, the only solution is $(j_1,j_2,j_3)=(1,1,0)$, whence
$$ \ker A(\lambda_0) = V^{(1,1,0)} , $$
which has dimension
$$ \dim\ker A(\lambda_0) = 3\cdot 3\cdot 1 = 9 , $$
which is odd. Note that these eigenfunctions have zero mean, consistently with the choice of Banach spaces.
As for transversality, $A(\lambda)$ acts on each $V^{(j_1,j_2,j_3)}$ as multiplication by the scalar
$$ a(\lambda) = 4\pi\lambda^{1/3}\left(j_1(j_1+1)+j_2(j_2+1)-\frac{4}{3}\right) + 4\pi\lambda^{-2/3}\left(j_3(j_3+1)-\frac{2}{3}\right) . $$
On $\ker A(\lambda_0)=V^{(1,1,0)}$ this reads $a(\lambda)=\frac{32\pi}{3}\lambda^{1/3}-\frac{8\pi}{3}\lambda^{-2/3}$, whence
$$ a'(\lambda) = \frac{32\pi}{9}\,\lambda^{-2/3} + \frac{16\pi}{9}\,\lambda^{-5/3} , \qquad a'(\lambda_0) = \frac{32\pi}{3}\,\lambda_0^{-2/3} \neq 0 . $$
Hence $\partial_\lambda A(\lambda_0)[v]=a'(\lambda_0)\,v$ is a non-zero multiple of $v$ for every $v\in\ker A(\lambda_0)\setminus\{0\}$; since $A(\lambda_0)$ is self-adjoint, $\mathrm{im}\,A(\lambda_0)=\left(\ker A(\lambda_0)\right)^{\perp}$, and therefore $\partial_\lambda A(\lambda_0)[v]\notin\mathrm{im}\,A(\lambda_0)$.

By the bifurcation theorem, for $\lambda$ close to $\lambda_0$ the equation $F(f,\lambda)=0$ admits a non-trivial solution $f_\lambda\not\equiv0$, and the metric $e^{\frac{f_\lambda}{3}}\omega_\lambda$ has constant Chern scalar curvature but is not equivalent to $\omega_\lambda$ under the action of $\mathrm{Aut}(X,\{\omega_\lambda\})\times\mathbb{R}^{>0}$.

Despite the failure of uniqueness, compactness of $\mathcal{C}h\mathcal{Y}a(X,\{\omega\})$ in the positive regime remains an open problem.
\end{remark}

\begin{remark}
The parabolic counterpart of the Chern-Yamabe equation is the Chern-Yamabe flow
$$\frac{\partial f}{\partial t} = -\Delta^{\mathrm{Ch}}_\omega f - \mathrm{scal}^{\mathrm{Ch}}(\omega) + \lambda\,e^{\frac{f}{n}},$$
where $\lambda=\Gamma_X(\{\omega\})$ is the expected constant Chern scalar curvature. Since the principal symbol of $\Delta^{\mathrm{Ch}}$ coincides with that of the Hodge-de Rham Laplacian, the flow is well-posed for any smooth initial datum. Short-time existence, long-time behaviour, and convergence in the negative Gauduchon degree case have been further studied in \cite{lejmi-maalaoui,calamai-zou,yu2}.
\end{remark}

\begin{remark}\label{rmk:second-chern-einstein}
The Chern-Yamabe problem has a natural application to the existence of second Chern-Einstein metrics, i.e., Hermitian metrics $\omega$ satisfying
$$\mathrm{Ric}^{(2)}(\omega) = \lambda\,\omega$$
for some smooth function $\lambda$; see \cite{gauduchon-cras-2,angella-calamai-spotti-2}. Recall that the second Chern-Ricci form $\mathrm{Ric}^{(2)}(\omega)$ is the $(1,1)$-form obtained by tracing the Chern curvature tensor in the first pair of indices,
$$\mathrm{Ric}^{(2)}(\omega) \stackrel{\mathrm{loc}}{=} h^{\bar{j}i}\,\Theta_{i\bar{j}k\bar{\ell}}\,\sqrt{-1}\,dz^\ell\wedge d\bar{z}^k ,$$
in contrast with the first Chern-Ricci form $\mathrm{Ric}^{\rm Ch}(\omega)$ recalled in Section~\ref{sec:preliminaries}.

Under a conformal change $\omega'=e^f\omega$, the second Chern-Ricci curvature transforms as
$$ \mathrm{Ric}^{(2)}(e^f\omega) = \mathrm{Ric}^{(2)}(\omega) + \frac{1}{2}(\Delta^{\mathrm{Ch}} f)\,\omega , $$
so the second Chern-Einstein condition is conformally invariant. Suppose $\omega$ is a \emph{weak} second Chern-Einstein metric, meaning that $\lambda$ is not necessarily constant. Solving the linear equation
$$\Delta^{\mathrm{Ch}}f = -\mathrm{scal}^{\mathrm{Ch}}(\eta) + \Gamma_X(\{\omega\})$$
(where $\eta$ is the unit-volume Gauduchon representative) produces a conformal metric $e^{\frac{f}{n}}\eta$ that is weak second Chern-Einstein with Einstein factor of definite sign, equal to the sign of $\Gamma_X(\{\omega\})$.

Recall that ${\rm tr}_\omega \mathrm{Ric}^{(2)}(\omega)=\frac{1}{2}{\rm scal}^{\rm Ch}(\omega)$. Finding a \emph{strong} second Chern-Einstein metric (i.e., with $\lambda$ constant) in the same conformal class reduces to solving the Chern-Yamabe equation \eqref{eq:ChYa}, namely
$$\Delta^{\mathrm{Ch}}_\eta f + \mathrm{scal}^{\mathrm{Ch}}(\eta) = \lambda\,e^{\frac{f}{n}} ,$$
and the resulting metric $e^{\frac{f}{n}}\eta$ has constant second Chern-Einstein factor $\frac{\lambda}{2n}$.
Therefore, a positive solution to the Chern-Yamabe problem in every conformal class would imply that every weak second Chern-Einstein metric is conformal to a strong one.
\end{remark}

The positive curvature regime raises several open questions: whether $\mathcal{C}h\mathcal{Y}a(X,\{\omega\})$ is non-empty for every conformal class with $\Gamma_X(\{\omega\})>0$, whether it is compact, whether the Chern-Yamabe functional is bounded below in the balanced case, and whether the Chern-Yamabe flow converges. A deeper understanding of the Chern scalar curvature requires a more geometric perspective: in Section~\ref{sec:momentum}, we show that a deformation of $\mathrm{scal}^{\mathrm{Ch}}$ by torsion quantities arises as a momentum map in locally conformally K\"ahler geometry. This picture then suggests a natural deformation of the Chern-Yamabe problem that is both geometrically motivated and variational, which we study in Section~\ref{sec:twisted-yamabe}.
\section{Chern-scalar curvature in the momentum map picture}
\label{sec:momentum}

Constant scalar curvature metrics play a central role in K\"ahler geometry: by the foundational work of Fujiki \cite{fujiki} and Donaldson \cite{donaldson}, the scalar curvature arises as a momentum map for an infinite-dimensional Hamiltonian action, so that constant scalar curvature K\"ahler metrics appear as its zeroes. Inspired by the Hilbert-Mumford criterion in Geometric Invariant Theory in finite dimension, this perspective leads to the notion of $K$-stability.

More precisely, let $X$ be a compact smooth manifold endowed with a symplectic form $\omega$ admitting K\"ahler structures. The space $\mathcal{J}_{\mathrm{alm}}(\omega)$ of $\omega$-compatible almost-complex structures carries a natural infinite-dimensional K\"ahler structure $(\mathbb{J},\mathbb{\Omega})$, invariant under the symplectomorphism group $\mathrm{Aut}(X,\omega)$, which preserves the subset $\mathcal{J}(\omega)$ of integrable structures. See Section~\ref{sec:fujiki-donaldson} for details. The scalar curvature map $\mathrm{scal}\colon\mathcal{J}_{\mathrm{alm}}(\omega)\to\mathcal{C}^\infty(X,\mathbb{R})$, assigning to $J$ the (Riemannian) scalar curvature of $g_J:=\omega(\cdot,J\cdot)$, is $\mathrm{Aut}(X,\omega)$-equivariant and satisfies the momentum map identity
$$\int_X d\,\mathrm{scal}|_J(v)\,h_Y\,\frac{\omega^n}{n!} = -\frac{1}{2}\,\mathbb{\Omega}_J(Y^*_J,v)$$
for all $J\in\mathcal{J}(\omega)$, $v\in T_J\mathcal{J}_{\mathrm{alm}}(\omega)$, and $Y\in\mathfrak{ham}(X,\omega)$, where $\mathfrak{ham}(X,\omega)$ consists of vector fields $Y$ satisfying $Y\lrcorner\omega=dh_Y$ for some $h_Y\in\mathcal{C}^\infty(X,\mathbb{R})$, unique up to a constant, and $Y^*$ denotes the fundamental vector field of the action.

In this section, following \cite{angella-calamai-pediconi-spotti}, we describe how this picture extends to the locally conformally K\"ahler setting, and how the Chern scalar curvature naturally enters deformed by a tensor term. It is this extension that motivates the deformed scalar curvatures of Section~\ref{sec:twisted-yamabe}.

Throughout this section, $X$ denotes a compact, connected, smooth manifold of real dimension $2n$, endowed with a locally conformally symplectic form $\omega$, namely a non-degenerate $2$-form such that
$$ d\omega = \frac{1}{n-1}\,\theta\wedge\omega , \qquad d\theta = 0 , $$
for some closed $1$-form $\theta$. Thanks to the coefficient ``$1/(n-1)$'' chosen here, $\theta$ coincides exactly with the Lee form introduced before. We assume $\theta$ to be non-exact, a case sometimes referred to as strictly locally conformally symplectic.

\subsection{The tautological K\"ahler structure on the space of locally conformally K\"ahler structures}\label{sec:fujiki-donaldson}

We first construct a K\"ahler structure on the infinite-dimensional space of locally conformally K\"ahler structures.
We stress that the construction of such structure, as in \cite{fujiki, donaldson} for the K\"ahler setting, requires only that $\omega$ be non-degenerate, not that it be closed or locally conformally symplectic.
Let $\omega$ be a non-degenerate $2$-form on $X$ and set
$$ \mathcal{J}_{\mathrm{alm}}(\omega) := \left\{ \begin{array}{c} J \text{ almost-complex structure on } X : \\
\omega(J\cdot,J\cdot)=\omega \text{ and } g_J:=\omega(\cdot,J\cdot)>0 \end{array} \right\} . $$
Following \cite[Remark 4.3]{fujiki}, $\mathcal{J}_{\mathrm{alm}}(\omega)$ has a natural structure of infinite-dimensional ILH (inverse limit of Hilbert) manifold in the sense of \cite{omori}, in particular Fr\'echet, being the space of smooth sections of a bundle with fibre ${\rm Sp}(2n,\mathbb R) / {\rm U}(n)$.

The tangent space of $\mathcal{J}_{\mathrm{alm}}(\omega)$ at $J$ is
\begin{eqnarray*}
T_J\mathcal{J}_{\mathrm{alm}}(\omega)
&=& \left\{ A\in\mathcal{C}^\infty(X,\mathfrak{sp}(TX,\omega)) : JA+AJ=0 \right\} \\
&=& \left\{ \hat{a}_J:=[J,a] : a\in\mathcal{C}^\infty(X,\mathfrak{sp}(TX,\omega)) \right\} ,
\end{eqnarray*}
where
$$ \mathfrak{sp}(TX,\omega) := \left\{ a\in\mathrm{End}(TX) : \omega(a\cdot,\cdot)+\omega(\cdot,a\cdot)=0 \right\} $$
denotes the bundle of infinitesimally $\omega$-preserving endomorphisms.
Moreover, each such $A$ is $g_J$-symmetric.

Indeed, let $(J_t)_t$ be a curve in $\mathcal{J}_{\mathrm{alm}}(\omega)$ with $J_0=J$ and $A:=\dot{J}_0:=\frac{d}{dt}\big|_{t=0}J_t$. Differentiating $J_t^2=-\mathrm{id}$ gives $JA+AJ=0$, and differentiating $\omega(J_t\cdot,J_t\cdot)=\omega$ gives $\omega(A\cdot,J\cdot)+\omega(J\cdot,A\cdot)=0$; using $J$-compatibility, the latter is equivalent to $A\in\mathfrak{sp}(TX,\omega)$. This proves the first inclusion.
Given $A$ with $JA+AJ=0$ and $A\in\mathfrak{sp}(TX,\omega)$, set $a:=-\tfrac{1}{2}JA$. Then $[J,a]=\tfrac{1}{2}A+\tfrac{1}{2}JAJ=A$, since $JAJ=A$; and $a\in\mathfrak{sp}(TX,\omega)$.
Conversely, for any $a\in\mathcal{C}^\infty(X,\mathfrak{sp}(TX,\omega))$, the curve $J_t:=\exp(-ta)\,J\,\exp(ta)$ lies in $\mathcal{J}_{\mathrm{alm}}(\omega)$, as $\exp(ta)$ preserves $\omega$, with $J_0=J$ and $\dot{J}_0=[J,a]$. Finally, $g(AX,Y)=\omega(AX,JY)=-\omega(X,AJY)=\omega(X,JAY)=g(X,AY)$, so $A$ is $g$-symmetric.

Any $a\in\mathcal{C}^\infty(X,\mathfrak{sp}(TX,\omega))$ thus determines a vector field $\hat{a}$ on $\mathcal{J}_{\mathrm{alm}}(\omega)$, with $\hat{a}_J:=[J,a]$; we call such vector fields ``basic''. Their flow is $\Theta^{\hat{a}}_t(J)=\exp(-ta)\,J\,\exp(ta)$.
Since $\hat a_J=[J,a]$ depends linearly on $J$, its differential at $J$ in the
direction $v$ is simply $d_J\hat a(v)=[v,a]$; the bracket of two basic vector fields
is then $[\hat a,\hat b]_J=d_J\hat b(\hat a_J)-d_J\hat a(\hat b_J)=[[J,a],b]-[[J,b],a]$, which by the Jacobi identity equals $[J,[a,b]]=\widehat{[a,b]}_J$.This gives the bracket identity
\begin{equation}\label{eq:basic-bracket}
[\hat{a},\hat{b}] = \widehat{[a,b]} = \widehat{ab-ba} .
\end{equation}

We now equip $\mathcal{J}_{\mathrm{alm}}(\omega)$ with the tautological triple $(\mathbb{J},\mathbb{G},\mathbb{\Omega})$, defined at $J$ by
\begin{eqnarray*}
\mathbb{J}_J(A) &:=& JA , \\
\mathbb{G}_J(A,B) &:=& \int_X\mathrm{tr}(AB) \,\frac{\omega^n}{n!} , \\
\mathbb{\Omega}_J(A,B) &:=& \mathbb{G}_J(\mathbb{J}_JA,B)=\int_X\mathrm{tr}(JAB)\,\frac{\omega^n}{n!} .
\end{eqnarray*}
One checks directly that $\mathbb{J}_J$ maps $T_J\mathcal{J}_{\mathrm{alm}}(\omega)$ to itself: if $JA+AJ=0$, then $J(JA)+(JA)J=-A+JAJ=-A+(-AJ)J=0$; and if moreover $A\in\mathfrak{sp}(TX,\omega)$, then $\omega(JAX,Y)+\omega(X,JAY) = \omega(X,AJY)-\omega(JX,AY) = -\omega(X,JAY)-\omega(JX,AY) = \omega(JX,AY)-\omega(JX,AY) = 0$, so $JA\in\mathfrak{sp}(TX,\omega)$ as well; hence $\mathbb{J}_J^2=-\mathrm{id}$ follows from $J^2=-\mathrm{id}$. Antisymmetry of $\mathbb\Omega_J$, i.e.\ $\mathbb\Omega_J(A,B)=-\mathbb\Omega_J(B,A)$, follows from $\mathrm{tr}(JAB)=\mathrm{tr}((JAB)^t)=-\mathrm{tr}(BAJ)=-\mathrm{tr}(JBA)$, using $A^t=A$, $B^t=B$, $J^t=-J$.

\begin{proposition}[{\cite[Theorem 4.2]{fujiki}}]\label{prop:tautological-kahler}
The triple $(\mathbb{J},\mathbb{G},\mathbb{\Omega})$ is a K\"ahler structure on $\mathcal{J}_{\mathrm{alm}}(\omega)$, with respect to which the basic vector fields are holomorphic and Killing.
\end{proposition}

\begin{proof}
The key point of the proof is the pointwise trace identity
\begin{equation}\label{eq:trace-identity}
\mathrm{tr}\!\left(J[J,a][J,b]\right) = 2\,\mathrm{tr}(J[a,b]) ,
\end{equation}
proved as follows. Writing $X:=[J,a]$, $Y:=[J,b]$, one computes $JX=J^2a-JaJ=-a-JaJ$, whence $JXY=-aJb+abJ+Jab+JaJbJ$. Taking traces and using cyclicity, $\mathrm{tr}(abJ)=\mathrm{tr}(Jab)$ and $\mathrm{tr}(JaJbJ)=\mathrm{tr}(J^2aJb)=-\mathrm{tr}(aJb)=-\mathrm{tr}(Jba)$; summing these gives $\mathrm{tr}(JXY)=2\,\mathrm{tr}(Jab)-2\,\mathrm{tr}(Jba)=2\,\mathrm{tr}(J[a,b])$, as claimed.

In particular,
$$ \mathbb{\Omega}_J(\hat{a}_J,\hat{b}_J)=2\int_X\mathrm{tr}(J[a,b])\frac{\omega^n}{n!} . $$
We show that $\mathbb{\Omega}$ is closed. Recall from the computation of
$T_J\mathcal{J}_{\mathrm{alm}}(\omega)$ above that every $A\in T_J\mathcal{J}_{\mathrm{alm}}(\omega)$
is of the form $\hat a_J$ for $a:=-\tfrac12 JA$; hence the basic vector fields span
each tangent space, and since $d\mathbb\Omega$ is $C^\infty$-tensorial, it suffices to
evaluate it on basic vector fields, whose brackets are well-defined by
\eqref{eq:basic-bracket}:
\[
d\mathbb{\Omega}(\hat{a},\hat{b},\hat{c}) = \sum_{\mathrm{cyc}}\left( \hat{a}\,\mathbb{\Omega}(\hat{b},\hat{c}) - \mathbb{\Omega}([\hat{a},\hat{b}],\hat{c}) \right) .
\]
Differentiating $\mathbb{\Omega}_{J_t}(\hat{b},\hat{c})=2\int_X\mathrm{tr}(J_t[b,c])\frac{\omega^n}{n!}$ along the flow $J_t=\Theta^{\hat{a}}_t(J)$ of $\hat a$ gives
$$ \hat{a}\,\mathbb{\Omega}(\hat{b},\hat{c})\big|_J = 2\int_X\mathrm{tr}\!\left([J,a][b,c]\right)\frac{\omega^n}{n!} = 2\int_X\mathrm{tr}\!\left(J\,[a,[b,c]]\right)\frac{\omega^n}{n!} , $$
where the second equality is a cyclic rearrangement under the trace. Summing cyclically, the first group vanishes by the Jacobi identity. By \eqref{eq:basic-bracket}, we get $\mathbb{\Omega}([\hat{a},\hat{b}],\hat{c})|_J=2\int_X\mathrm{tr}(J[[a,b],c])\frac{\omega^n}{n!}$, and the second group vanishes by the Jacobi identity as well. Hence $d\mathbb{\Omega}=0$.

Integrability of $\mathbb{J}$ follows from an analogous computation, showing that the
Nijenhuis tensor vanishes on basic vector fields, via $[\hat{a},\mathbb{J}\hat{b}]=\mathbb{J}[\hat{a},\hat{b}]$.
Indeed, $(\mathbb{J}\hat{b})_J=J[J,b]=-b-JbJ$, so that $d_J(\mathbb{J}\hat b)(v)=-(vbJ+Jbv)$; substituting $v=\hat a_J=[J,a]$ and $d_J\hat a\big((\mathbb J\hat b)_J\big)=[(\mathbb J\hat b)_J,a]$ into $[\hat a,\mathbb{J}\hat b]_J=d_J(\mathbb J\hat b)(\hat a_J)-d_J\hat a\big((\mathbb J\hat b)_J\big)$
and simplifying gives $[\hat a,\mathbb J\hat b]_J=-[a,b]-J[a,b]J=\mathbb{J}_J\big(\widehat{[a,b]}_J\big)=\mathbb{J}[\hat a,\hat b]_J$.
Since basic vector fields span each tangent space, this shows that the Nijenhuis tensor of $\mathbb{J}$ vanishes identically on $\mathcal{J}_{\mathrm{alm}}(\omega)$; by \cite[Proposition~1.4]{koiso}, in this infinite-dimensional setting the vanishing of the Nijenhuis tensor is in fact equivalent to the integrability of $\mathbb{J}$, so this suffices to conclude.

Compatibility is immediate: for $A,B\in T_J\mathcal{J}_{\mathrm{alm}}(\omega)$ one has $\mathrm{tr}(JAJB)=\mathrm{tr}(AB)$, hence $\mathbb{G}_J(\mathbb{J}A,\mathbb{J}B)=\mathbb{G}_J(A,B)$; and $\mathbb{G}_J(A,A)=\int_X\mathrm{tr}(A^2)\frac{\omega^n}{n!}\geq0$ since $A$ is $g$-symmetric, with equality only for $A=0$.
\end{proof}

Finally, we introduce the subset
$$ \mathcal{J}(\omega) := \left\{ J \in \mathcal{J}_{\mathrm{alm}}(\omega) : J \text{ is integrable} \right\} , $$
which is an analytic subset of $\mathcal{J}_{\mathrm{alm}}(\omega)$.

\subsection{Twisted Hamiltonian vector fields}

We now bring in the locally conformally symplectic condition $d\omega=\theta\wedge\omega$, $d\theta=0$. Introducing the twisted differential
$$ d_\theta := d - \frac{1}{n-1}\theta\wedge , $$
the condition reads $d_\theta\omega=0$; since $\theta$ is closed, $d_\theta\circ d_\theta=0$, so $d_\theta$ defines the Morse-Novikov cohomology $H^\bullet_{d_\theta}(X,\mathbb{R})$. We assume $\theta$ to be non-exact, which we call the strictly locally conformally symplectic case.
Then $H^0_{d_\theta}(X,\mathbb{R})=\{0\}$: indeed, if $h$ satisfies $dh=\frac{1}{n-1}h\theta$ and vanishes at some point $p$, then along any path $\gamma$ from $p$ the function $h\circ\gamma$ solves a homogeneous linear ordinary differential equation with zero initial datum, forcing $h\circ\gamma\equiv0$; since $X$ is connected, $h\equiv0$. Otherwise $h$ is nowhere zero, hence of constant sign by continuity on connected $X$, and (replacing $h$ by $-h$ if necessary) $\log h$ is a well-defined global smooth function with $d\log h=\theta/(n-1)$, so that $\theta=(n-1)\,d\log h$ is exact, contradicting our assumption. Hence the only function $h$ with $dh=\frac{1}{n-1}\,h\theta$ is $h=0$.
On the minimal covering $\pi\colon\widehat{X}\to X$, the pullback $\pi^*\omega$ is globally conformal to a symplectic form, and the deck group acts by homotheties.
Any compatible complex structure $J \in \mathcal{J}(\omega)$ corresponds to a locally conformally K\"ahler structure $\left(J, g_J:=\omega(\cdot, J \cdot)\right)$.

The natural symmetries are the special conformal vector fields and the twisted-Hamiltonian vector fields,
\begin{eqnarray*}
\mathfrak{aut}^\star(X,\{\omega\}) &:=& \left\{ Y : d_\theta(Y\lrcorner\omega)=0 \right\} , \\
\mathfrak{ham}(X,\{\omega\}) &:=& \left\{ Y : Y\lrcorner\omega=d_\theta h \text{ for some } h\in\mathcal{C}^\infty(X,\mathbb{R}) \right\} .
\end{eqnarray*}
The reason for these definitions is explained by the Cartan formula:
\begin{equation}\label{eq:cartan-twisted}
\mathcal{L}_Y\omega = d(Y\lrcorner\omega) + Y\lrcorner(\theta\wedge\omega) = d_\theta(Y\lrcorner\omega) + \frac{1}{n-1}\,\theta(Y)\,\omega ,
\end{equation}
so that $Y\in\mathfrak{aut}^\star(X,\{\omega\})$ if and only if $\mathcal{L}_Y\omega=\frac{1}{n-1}\,\theta(Y)\,\omega$.
A computation with the twisted Cartan formula shows that both are closed Lie subalgebras of the algebra $\mathfrak{X}(X)$ of vector fields, with $[\mathfrak{aut}^\star,\mathfrak{aut}^\star]\subseteq\mathfrak{ham}\subseteq\mathfrak{aut}^\star$. The inclusion $\mathfrak{ham}\subseteq\mathfrak{aut}^\star$ is immediate from $d_\theta^2=0$. For the first inclusion, given $Y_1,Y_2\in\mathfrak{aut}^\star$, the identity $\iota_{[Y_1,Y_2]}=[\mathcal{L}_{Y_1},\iota_{Y_2}]$ gives
\begin{eqnarray*}
[Y_1,Y_2]\lrcorner\omega
&=& \mathcal{L}_{Y_1}(Y_2\lrcorner\omega) - \frac{1}{n-1}\theta(Y_1)\cdot(Y_2\lrcorner\omega) \\
&=& d \left( Y_1 \lrcorner (Y_2\lrcorner\omega) \right) + Y_1 \lrcorner d(Y_2\lrcorner\omega) - \frac{1}{n-1}\theta(Y_1)\cdot(Y_2\lrcorner\omega) \\
&=& d \left( \omega(Y_2,Y_1) \right) + \frac{1}{n-1} Y_1 \lrcorner \left( \theta \wedge (Y_2 \lrcorner \omega) \right) - \frac{1}{n-1}\theta(Y_1)\cdot(Y_2\lrcorner\omega) \\
&=& d \left( \omega(Y_2,Y_1) \right) + \frac{1}{n-1} \theta(Y_1) \cdot (Y_2 \lrcorner \omega) - \frac{1}{n-1} \omega(Y_2,Y_1) \cdot \theta \\
&& - \frac{1}{n-1}\theta(Y_1)\cdot(Y_2\lrcorner\omega) \\
&=& d \left( \omega(Y_2,Y_1) \right) - \frac{1}{n-1} \omega(Y_2,Y_1) \cdot \theta \\
&=& d_\theta\big(\omega(Y_2,Y_1)\big) ,
\end{eqnarray*}
which in particular shows $[Y_1,Y_2]\in\mathfrak{ham}(X,\{\omega\})$.

Crucially, since $H^0_{d_\theta}(X,\mathbb{R})=\{0\}$, the twisted-Hamiltonian potential is unique: there is a linear isomorphism
$$ \kappa\colon \mathfrak{ham}(X,\{\omega\})\xrightarrow{\ \sim\ }\mathcal{C}^\infty(X,\mathbb{R}) , \qquad Y\lrcorner\omega=d_\theta(\kappa(Y)) , $$
which satisfies $\kappa([Y_1,Y_2])=\omega(Y_2,Y_1)$.
This is in sharp contrast with the symplectic case, where the potential is defined only up to constants.

We denote by $\mathrm{Aut}^\star(X,\{\omega\})$ and $\mathrm{Ham}(X,\{\omega\})$ the generalized Lie transformation groups generated by these algebras, in the sense of \cite{omori}. The group $\mathrm{Aut}^\star(X,\{\omega\})$ acts on $\mathcal{J}_{\mathrm{alm}}(\omega)$ by pullback, $\phi\cdot J:=(d\phi)^{-1}\circ J\circ d\phi$; this preserves the integrable locus $\mathcal{J}(\omega)$, since the Nijenhuis tensor transforms equivariantly, $N_{\phi\cdot J}=\phi^*N_J$. The fundamental vector field of $Y\in\mathfrak{aut}^\star(X,\{\omega\})$ is given by $Y^* = -\mathcal{L}_Y J$; the sign is the standard normalization ensuring $Y\mapsto Y^*$ is a Lie algebra homomorphism.

\subsection{The momentum map}
\label{subsec:moment-map}

We would like to mimic the Fujiki-Donaldson construction for the action of $\mathrm{Ham}(X,\{\omega\})$ on $\mathcal{J}(\omega)$, using the Chern scalar curvature, which makes the Riemannian scalar curvature $J\mapsto\mathrm{scal}(g_J)$ appear by \eqref{eq:scal-scalch}. Two difficulties arise:
\begin{itemize}
\item the group $\mathrm{Aut}^\star(X,\{\omega\})$ preserves only the conformal class $\{\omega\}$, since $\mathcal{L}_Y\omega=\frac{1}{n-1}\theta(Y)\,\omega$ for $Y\in\mathfrak{aut}^\star(X,\{\omega\})$; hence the map $J\mapsto\mathrm{scal}(g_J)$ is not equivariant;
\item as we will see below, the linearization of $\mathrm{scal}$ produces a term involving $\nabla(\theta^\sharp)$, where $\nabla$ denotes the Weyl connection, which cannot apparently be absorbed by corrections depending on $\theta$.
\end{itemize}

Both difficulties are overcome by imposing an extra symmetry. Consider the symplectic dual of the multiple of the Lee form, namely the vector field $V$ defined by $V\lrcorner\omega=\frac{1}{n-1}\theta$. Since $\frac{1}{n-1}\theta=d_\theta(-1)$, the field $V$ is twisted-Hamiltonian with potential $\kappa(V)=-1$, and $\mathcal{L}_V\omega=0$. Note also that $\theta^\sharp=(n-1)JV$. We assume that there exists a compact, connected Lie group $K$ acting effectively on $X$ such that:
\begin{itemize}
\item the action preserves $\omega$;
\item the action is of Lee type, in the sense that $V\in\mathfrak{z}(\mathfrak{k})$ belongs to the center of the Lie algebra $\mathfrak{k}$ of $K$;
\item the action is twisted-Hamiltonian, i.e., $\mathfrak{k}\subseteq\mathfrak{ham}(X,\{\omega\})$.
\end{itemize}
The model case is that of Vaisman metrics (namely, locally conformally K\"ahler metrics whose Lee form is parallel, $D\theta=0$), where the closure of the flow of $V$ is a torus of Lee type. We restrict all objects to their $K$-invariant counterparts, denoted with a superscript $K$.

The symmetry solves exactly the two issues we had before. First, equivariance is restored: for $Y\in\mathfrak{aut}^\star(X,\{\omega\})^{K}$, the twisted Leibniz rule gives
$$ d_\theta\big(\omega(V,Y)\big) = V\lrcorner\,d_\theta(Y\lrcorner\omega) - Y\lrcorner\,d_\theta(V\lrcorner\omega) - [V,Y]\lrcorner\,\omega = 0 , $$
since $d_\theta(Y\lrcorner\omega)=0$, $d_\theta(V\lrcorner\omega)=d_\theta d_\theta(-1)=0$, and $[V,Y]=0$ (as $V$ is central and $Y$ is $K$-invariant). Since $H^0_{d_\theta}(X,\mathbb{R})=\{0\}$, this forces $\theta(Y)=0$, hence $\mathcal{L}_Y\omega=0$: the invariant group genuinely preserves $\omega$. Second, for $J\in\mathcal{J}(\omega)^{K}$ we have $\mathcal{L}_VJ=0$, which is equivalent to $[J,\nabla V]=0$; since $\nabla J=0$ and $\theta^\sharp=(n-1)JV$, the endomorphism $\nabla(\theta^\sharp)=(n-1)J\circ\nabla V$ is $J$-invariant, hence pointwise orthogonal to the $J$-anti-invariant tangent directions of $\mathcal{J}_{\mathrm{alm}}(\omega)^{K}$: the problematic term will drop out of all pairings.

We can now state the main result. Recall that $\mathrm{scal}^{\mathrm{Ch}}$ denotes the Chern scalar curvature of the Hermitian metric $g_J=\omega(\cdot,J\cdot)$.

\begin{theorem}[{\cite{angella-calamai-pediconi-spotti}}]\label{thm:moment-map}
Let $(X,\omega)$ be a compact, connected, locally conformally symplectic manifold of real dimension $2n$ with non-exact Lee form $\theta$, and let $K$ be a compact, connected Lie group acting effectively on $X$ and satisfying the conditions above. Then the map
$$ \mu\colon\mathcal{J}(\omega)^{K}\to\mathcal{C}^\infty(X,\mathbb{R})^{K} , \qquad \mu := \mathrm{scal}^{\mathrm{Ch}} + \frac{n}{n-1}\,d^*\theta , $$
is a momentum map for the action of $\mathrm{Ham}(X,\{\omega\})^{K}$ on $\mathcal{J}(\omega)^{K}$: more precisely, it admits a smooth extension to an $\mathrm{Aut}^\star(X,\{\omega\})^K$-equivariant  map on the whole space $\mathcal J_{\mathrm{alm}}(\omega)^K$ such that, for every $J\in\mathcal{J}(\omega)^{K}$, every $\hat{a}_J\in T_J\mathcal{J}_{\mathrm{alm}}(\omega)^{K}$, and every $Y\in\mathfrak{ham}(X,\{\omega\})^{K}$ with potential $h=\kappa(Y)$,
\begin{equation}\label{eq:moment-map-identity}
\int_X d\mu|_J(\hat{a}_J)\,h\,\frac{\omega^n}{n!} = -\frac{1}{2}\,\mathbb{\Omega}_J(Y^*_J,\hat{a}_J) .
\end{equation}
\end{theorem}

\begin{proof}
When $J$ is integrable, the metric $g_J$ is locally conformally K\"ahler, and the Chern scalar curvature is related to the Riemannian one by
$$ \mathrm{scal}^{\mathrm{Ch}} = \mathrm{scal} - d^*\theta + \frac{1}{2(n-1)}\,|\theta|^2 . $$
This follows from Proposition~\ref{prop:scalch-scalg} above (i.e.\ \cite[Equation~(33)]{gauduchon-mathann}), specialized to the locally conformally K\"ahler case, where $|T|^2=\frac{2}{n-1}|\theta|^2$.
Therefore
$$ \mu = \mathrm{scal} + \frac{1}{n-1}\,d^*\theta + \frac{1}{2(n-1)}\,|\theta|^2 , $$
and the right-hand side is defined for every $J\in\mathcal{J}_{\mathrm{alm}}(\omega)^{K}$: this is the expression we linearize. Note that all three summands depend on $J$ through the metric $g_J$, while $\omega$ and $\theta$ are fixed. Recall from the discussion above that $\mathrm{Aut}^\star(X,\{\omega\})^K$ preserves $\omega$ exactly (not merely the conformal class $\{\omega\}$, as a general special conformal vector field would): indeed, for $Y\in\mathfrak{aut}^\star(X,\{\omega\})^K$ we showed $\theta(Y)=0$, hence $\mathcal{L}_Y\omega=0$. Consequently $\varphi^*\omega=\omega$ for every $\varphi\in\mathrm{Aut}^\star(X,\{\omega\})^K$; since $\theta$ is itself determined by $\omega$, this forces $\varphi^*\theta=\theta$ as well. It follows that $g_{\varphi\cdot J}=\varphi^*g_J$ for any $J\in\mathcal{J}_{\mathrm{alm}}(\omega)^K$, so that the extended map $\mu$ is $\mathrm{Aut}^\star(X,\{\omega\})^K$-equivariant.

\smallskip
\noindent{\itshape Step 1. Linearization formula for the norm of the Lee form.}
Fix $J\in\mathcal{J}(\omega)^{K}$ and a direction $\hat{a}_J=[J,a]\in T_J\mathcal{J}_{\mathrm{alm}}(\omega)^{K}$, and set
$$ \mathring{a} := -J\hat{a}_J , $$
which is $g$-symmetric, $J$-anti-invariant, and trace-free (these three properties follow directly from $\hat a_J\in T_J\mathcal{J}_{\mathrm{alm}}(\omega)^K$). Denote by a prime the derivative at $J$ along $\hat{a}_J$. Differentiating $g_J=\omega(\cdot,J\cdot)$ gives
$$ g' = \omega(\cdot,\hat{a}_J\cdot) = g(\mathring{a}\,\cdot,\cdot) . $$
Since $\theta$ is fixed, differentiating $0=\theta'=(g(\theta^\sharp,\cdot))'=g(\mathring{a}\,\theta^\sharp,\cdot)+g((\theta^\sharp)',\cdot)$ gives $(\theta^\sharp)'=-\mathring{a}(\theta^\sharp)$, and consequently, writing $|\theta|^2=\theta(\theta^\sharp)$ with $\theta$ itself independent of $J$,
$$ (|\theta|^2)' = \theta\big((\theta^\sharp)'\big) = -g(\mathring{a},\theta\otimes\theta^\sharp) . $$

\smallskip
\noindent{\itshape Step 2. Linearization formula for the scalar curvature.}
For the Riemannian scalar curvature, the general variation formula (see \cite[Theorem 1.174]{besse}) together with $\mathrm{tr}(\mathring{a})=0$ gives
$$ \mathrm{scal}' = d^*(\delta\mathring{a}) - g(\mathring{a},\mathrm{Ric}^\sharp) , $$
where $\delta$ denotes the divergence on symmetric endomorphisms. This is where the symmetry intervenes.
To make it clear, it is useful to introduce the Weyl connection.
Recall that the Weyl connection associated to $(g,\theta)$ is the unique torsion-free connection preserving the conformal class of $g$ with Lee form $\theta$, explicitly
$$ \nabla^W_XY := D_XY - \frac{1}{2(n-1)}\Big(\theta(X)Y+\theta(Y)X-g(X,Y)\theta^\sharp\Big) ; $$
one checks directly that $\nabla^Wg=\frac{1}{n-1}\theta\otimes g$ and, crucially, in the locally conformally K\"ahler setting one has $\nabla^WJ=0$.
Writing the Riemannian Ricci tensor $\mathrm{Ric}$ in terms of the Weyl-Ricci tensor $\mathrm{Ric}^W$ (see e.g.~\cite[Proposition A.6]{angella-calamai-pediconi-spotti}), one has
$$ (\mathrm{Ric}^W)^\sharp-\mathrm{Ric}^\sharp = \nabla^W(\theta^\sharp) + \frac{1}{2(n-1)}\,\theta\otimes\theta^\sharp - \frac{1}{2(n-1)}\,(d^*\theta)\,\mathrm{id} . $$
Now, $(\mathrm{Ric}^W)^\sharp$ is symmetric and $J$-invariant (see e.g.~\cite[Corollary A.5]{angella-calamai-pediconi-spotti}), and $\nabla^W(\theta^\sharp)$ is $J$-invariant: indeed, writing $\theta^\sharp=(n-1)JV$ and using $\nabla^WJ=0$, one has $\nabla^W(\theta^\sharp)=(n-1)J\circ\nabla^WV$; since $V$ is central in $\mathfrak k$ and $J$ is $K$-invariant, $[J,\nabla^WV]=0$ as established above, whence $J\circ\nabla^W(\theta^\sharp)=-(n-1)\nabla^WV=\nabla^W(\theta^\sharp)\circ J$. Hence both terms are pointwise orthogonal to the $J$-anti-invariant $\mathring{a}$; the $\mathrm{id}$-term pairs to $\mathrm{tr}(\mathring{a})=0$.
Hence
$$ \mathrm{scal}' = d^*(\delta\mathring{a}) + \frac{1}{2(n-1)}\,g(\mathring{a},\theta\otimes\theta^\sharp) . $$

\smallskip
\noindent{\itshape Step 3. Linearization formula for the codifferential of the Lee form.}
Since $d^*\theta = -\mathrm{tr}\,D(\theta^\sharp)$, we get
$$ (d^*\theta)' = -\mathrm{tr}\big(D'(\theta^\sharp)\big) - \mathrm{tr}\big(D\big((\theta^\sharp)'\big)\big) . $$
For the first term, a computation on a local $g$-orthonormal frame $(\tilde e_\alpha)$, using \cite[Theorem~1.174]{besse} together with $g'=g(\mathring a\,\cdot,\cdot)$, gives $2\,\mathrm{tr}(D'(\theta^\sharp))=\mathrm{tr}(D_{\theta^\sharp}\mathring a)=\mathcal L_{\theta^\sharp}(\mathrm{tr}\,\mathring a)=0$, using $\mathrm{tr}(\mathring a)=0$.
For the second term, using $(\theta^\sharp)'=-\mathring a(\theta^\sharp)$ from Step~1,
$$ -\mathrm{tr}\big(D((\theta^\sharp)')\big) = \mathrm{tr}\big(D(\mathring a(\theta^\sharp))\big) = -(\delta\mathring a)(\theta^\sharp) + g\big(D(\theta^\sharp),\mathring a\big) . $$
A direct computation from the definition of the Weyl connection in Step~2 gives $D(\theta^\sharp) = \nabla^W(\theta^\sharp) + \frac12|\theta|^2\,\mathrm{id}$; since $\nabla^W(\theta^\sharp)$ is orthogonal to $\mathring a$ (Step~2) and $\mathrm{tr}(\mathring a)=0$, we get $g(D(\theta^\sharp),\mathring a)=0$. Combining the three vanishing contributions,
$$ (d^*\theta)' = -(\delta\mathring{a})(\theta^\sharp) . $$

\smallskip
\noindent{\itshape Step 4. Summing up to get the linearization for the momentum map.}
Summing the three formulas with the coefficients of Step 1, the terms in $g(\mathring{a},\theta\otimes\theta^\sharp)$ cancel:
\begin{eqnarray*}
\mu' &=& \mathrm{scal}' + \frac{1}{n-1}(d^*\theta)' + \frac{1}{2(n-1)}(|\theta|^2)' \\
&=& d^*(\delta\mathring{a}) - \frac{1}{n-1}\,(\delta\mathring{a})(\theta^\sharp) .
\end{eqnarray*}
Pairing with the potential $h$ and integrating by parts (via $\langle h,d^*\beta\rangle_{L^2}=\langle dh,\beta\rangle_{L^2}$ and $\langle\beta,\delta\mathring{a}\rangle_{L^2}=\langle\delta^*\beta,\mathring{a}\rangle_{L^2}$),
$$ \langle h,\mu'\rangle_{L^2} = \Big\langle\delta^*\Big(dh-\frac{h}{n-1}\,\theta\Big),\mathring{a}\Big\rangle_{L^2} = \langle\delta^*(d_\theta h),\mathring{a}\rangle_{L^2} . $$

\smallskip
\noindent{\itshape Step 5. Conclusion.}
Since $d_\theta h=Y\lrcorner\omega=g(JY,\cdot)$, we have $(d_\theta h)^\sharp=JY$. A computation with the Weyl connection (see e.g.~\cite[Lemma~1.3]{angella-calamai-pediconi-spotti}, translated to the present normalization) shows that, for any $1$-form $\zeta$,
$$ \delta^*\zeta = \nabla(\zeta^\sharp) - \frac{1}{2}\,(\cdot\lrcorner\, d_\theta\zeta)^\sharp + \frac{1}{2(n-1)}\,g(\zeta,\theta)\,\mathrm{id} ; $$
applying this to $\zeta=d_\theta h$ and using $d_\theta\circ d_\theta=0$ and $\nabla J=0$, we get
$$ \delta^*(d_\theta h) = J\circ\nabla Y + \frac{1}{2(n-1)}\,\theta(JY)\,\mathrm{id} . $$
Pairing with the trace-free $\mathring{a}$ kills the $\mathrm{id}$-term. Finally, differentiating $d_\theta(Y\lrcorner\omega)=0$ yields the identity $2\,\mathrm{Sym}_g(\nabla^W Y)=-J\circ\mathcal{L}_YJ$, valid for special conformal vector fields (see e.g.~\cite[Lemmata 1.3-1.4]{angella-calamai-pediconi-spotti}), and recalling that the fundamental vector field is $Y^*=-\mathcal{L}_YJ$, we conclude
\begin{eqnarray*}
\langle h,\mu'\rangle_{L^2}
&=& \langle J\circ\nabla^W Y,\mathring{a}\rangle_{L^2} = -\langle\nabla^W Y,\hat{a}_J\rangle_{L^2} \\
&=& -\frac{1}{2}\,\langle\mathcal{L}_YJ,\mathbb{J}\hat{a}_J\rangle_{L^2} = -\frac{1}{2}\,\mathbb{\Omega}_J(Y^*_J,\hat{a}_J) ,
\end{eqnarray*}
which is precisely \eqref{eq:moment-map-identity}.
\end{proof}

As in the K\"ahler case, the momentum map picture produces an obstruction of Futaki type. This follows the classical Fujiki-Donaldson argument for the Futaki invariant, using that the constant function $1$ is itself a twisted-Hamiltonian potential, $1=\kappa(-V)$, for the central field $-V\in\mathfrak{z}(\mathfrak{k})$. 

\begin{corollary}[{\cite{angella-calamai-pediconi-spotti}}]\label{cor:futaki}
Under the assumptions of Theorem~\ref{thm:moment-map}, the average $\underline{\mu}:=\int_X\mu(J)\,\frac{\omega^n}{n!}\,\big/\int_X\frac{\omega^n}{n!}$ and the linear map
$$ \mathcal{F}\colon\mathfrak{z}(\mathfrak{k})\to\mathbb{R} , \qquad \mathcal{F}(Y):=\int_X\big(\mu(J)-\underline{\mu}\big)\,\kappa(Y)\,\frac{\omega^n}{n!} , $$
are independent of $J$ within each connected component of $\mathcal{J}(\omega)^{K}$. In particular, the non-vanishing of $\mathcal{F}$ obstructs the existence of $K$-invariant complex structures with $\mu(J)$ constant.
\end{corollary}

\begin{remark}\label{rmk:momentum-covering}
Two comments on the equation $\mu(J)=\underline{\mu}$. First, it does \emph{not} reduce to the constant scalar curvature K\"ahler equation on the minimal covering: the scalar curvature of the K\"ahler metric upstairs differs from the pullback of $\mu$ by terms involving $d^*\theta$ and $|\theta|^2$. Second, in the Vaisman case one has $D(\theta^\sharp)=0$, where $D$ is the Levi-Civita connection, hence $d^*\theta=0$, and the condition reduces to the constant Chern scalar curvature equation on $X$. In general, the above result suggests that the constancy of the momentum map, rather than of the Chern scalar curvature itself, is the natural notion of canonical metric in this setting: this motivates the deformed Yamabe problems of the next section.
\end{remark}

\begin{remark}
The momentum map picture developed in Theorem~\ref{thm:moment-map} naturally raises the question of defining a suitable notion of stability in the locally conformally K\"ahler setting, in the spirit of the Yau-Tian-Donaldson conjecture in K\"ahler geometry.
For recent results in the Sasaki context, see \cite{apostolov-calderbank-legendre, apostolov-jubert-lahdili, lahdili-legendre-scarpa}.
Compare also the work of Goto \cite{goto}, who develops a momentum map framework in generalized K\"ahler geometry (in the sense of Hitchin and Gualtieri) using the pure spinor formalism.
\end{remark}
\section{Deformed Yamabe problems}
\label{sec:twisted-yamabe}

Motivated by the results of the previous sections, we introduce and study a family of deformations of the Yamabe problem on compact complex manifolds. The deformation is defined by adding natural torsion terms to the Riemannian scalar curvature, and includes both the classical Yamabe equation (that we tried to extend in Section~\ref{sec:chern-yamabe}) and the scalar curvature equation arising in locally conformally K\"ahler geometry as a momentum map (as studied in Section~\ref{sec:momentum}). The results of this chapter are obtained in joint work with Francesco Pediconi, Carlo Scarpa, Cristiano Spotti, and Joshua Windare \cite{angella-pediconi-scarpa-spotti-windare}.

\subsection{Deformed scalar curvatures}

Let $X^n$ be a compact complex manifold of complex dimension $n$. For every Hermitian metric $\omega$ on $X$ and for every $t > 0$, we consider the $t$-deformed scalar curvature of $\omega$, defined as
\begin{equation}\label{eq:mu_t}
\mu^{t}(\omega) := \mathrm{scal}^{\mathrm{Ch}}(\omega) +\frac{t-n}{n-1}d^*\theta + \frac{t-2n}{4}|T|^2 \,\, ,
\end{equation}
where $\mathrm{scal}^{\mathrm{Ch}}(\omega)$ denotes the Chern scalar curvature, $T$ is the torsion of the Chern connection, and $\theta$ is the Lee form of $\omega$.

We explain here the choice of the coefficients.
First, we note that some special values of the parameter recover interesting geometric situations:
\begin{itemize}
\item $t=2n-1$: the relation between the Chern scalar curvature ${\rm scal}^{\rm Ch}(\omega)$ of $\omega$ and the Riemannian scalar curvature ${\rm scal}(g)$ of the underlying Riemannian metric $g$ in \eqref{eq:scal-scalch} implies that the $(2n-1)$-deformed scalar curvature coincides with the Riemannian scalar curvature;
\item $t=2n$: the quantity appearing in Section~\ref{sec:momentum} as a momentum map in the locally conformally K\"ahler setting, under the additional symmetry assumptions described there, coincides with the $2n$-deformed scalar curvature, thus providing a geometric interpretation for the latter.
\end{itemize}

Second, we recall that we are ultimately interested in studying the existence of metrics with constant deformed scalar curvature in a given conformal class, and hence in the behaviour of the deformed scalar curvature under conformal transformations. Consider more generally a quantity of the form
$$ \psi^{a,b}(\omega) := \mathrm{scal}^{\mathrm{Ch}}(\omega) +a\,d^*\theta +b\,|T|^2 $$
for parameters $a,b\in\mathbb R$.
By \eqref{eq:scal-conf}, \eqref{eq:theta-conf}, \eqref{eq:T-conf} in Proposition~\ref{prop:conf}, we get that
\begin{eqnarray}\label{eq:conformal-psi-a-b}
e^{f}\psi^{a,b}(e^{f}\omega)
&=& \psi^{a,b}(\omega) +(n +(n-1)a)\Delta f \\
&& + (n -(n-1)a +4b)g(\nabla f,\theta^\sharp) \nonumber \\
&& -(n-1)((n-1)a - 2b)|\nabla f|^2 . \nonumber
\end{eqnarray}
In particular, the choice in \eqref{eq:mu_t} is made so that the coefficient $n -(n-1)a +4b$ of the gradient term vanishes. As mentioned in Remark~\ref{rmk:conformal-formulas-not-variational}, see Remark~\ref{rmk:not-variational} below, this coefficient is the obstruction to the equation being the Euler-Lagrange equation of a suitable functional.

Third, we comment on the restriction $t>0$ rather than allowing general $t\in\mathbb R$. The value $t=0$ leads to a degenerate situation: the conformal transformation rule becomes
\begin{equation}\label{eq:conf-mu0}
e^f \mu^0(e^f\omega) = \mu^0(\omega) ,
\end{equation}
that is, $\mu^0$ is conformally equivariant. Consequently, the Yamabe-type equation $\mu^0(e^f\omega)=c$ is solvable whenever $\mu^0(\omega)$ has the same sign as $c$. Note however that the sign of $\mu^0(\omega)$ does play a role in the existence theory for constant $t$-deformed scalar curvature metrics with $t>0$ (see Theorem~\ref{thm:deformed-yamabe-existence}).

\subsection{Deformed Yamabe equations}

We are interested in the Yamabe problem for the $t$-deformed scalar curvature, that is, in finding Hermitian metrics with constant $t$-deformed scalar curvature within a fixed conformal class. Given $\omega$, we seek $f$ such that $\mu^t(e^f\omega) = c$ for some constant $c \in \mathbb{R}$.

Setting $a=\frac{t-n}{n-1}$ and $b=\frac{t-2n}{4}$ in \eqref{eq:conformal-psi-a-b}, we obtain the conformal transformation rule for the deformed scalar curvature:
\begin{equation}\label{eq:conf-mu-t}
e^{f}\,\mu^t(e^f\omega) = \mu^t(\omega) + t\Delta f -\frac{t(n-1)}{2} |\nabla f|^2 .
\end{equation}
It is useful to perform the following change of variables. We set
$$ f = \frac{2}{n-1}\log u, \qquad u = e^{\frac{n-1}{2}f} . $$
A straightforward computation gives
$$
\Delta f = \frac{2}{n-1}\left(\frac{\Delta u}{u}+\frac{|\nabla u|^2}{u^2}\right) \,\, , \quad
|\nabla f|^2 = \frac{4}{(n-1)^2}\frac{|\nabla u|^2}{u^2} \,\, .
$$
Substituting these identities into the conformal transformation formula, we obtain
$$
u^{\frac{2}{n-1}} \mu^t\left(u^{\frac{2}{n-1}}\omega\right)
=
\mu^t(\omega) + \frac{2t}{n-1}\frac{\Delta u}{u}
$$
and hence $\mu^t(e^f\omega) = c$ if and only if $u$ is a smooth positive solution of the semilinear equation
\begin{equation}\label{eq:t-DY}\tag{$t$-DY}
\frac{2t}{n-1}\Delta u + \mu^t(\omega)\,u = c\, u^{\frac{n+1}{n-1}} .
\end{equation}
We will refer to it as the $t$-deformed Yamabe equation.

The exponent $\frac{n+1}{n-1}$ of $u$ in \eqref{eq:t-DY} plays a fundamental role.
Recall that the Sobolev embedding theorem gives $W^{1,p}(\mathbb{R}^{m})\hookrightarrow L^{p^*}(\mathbb{R}^{m})$, where $\frac{1}{p^*}=\frac{1}{p}-\frac{1}{m}$, see e.g. \cite{aubin,talenti}. In our setting, we have $p=2$ and $m=2n$, so the critical Sobolev exponent for $W^{1,2}(\mathbb{R}^{2n})\hookrightarrow L^{2^*}(\mathbb{R}^{2n})$ is
$$ 2^* = \frac{2 \cdot (2n)}{(2n)-2} = \frac{2n}{n-1} , $$
and one checks that $\frac{n+1}{n-1}=2^*-1$. Thus \eqref{eq:t-DY} is a critical elliptic partial differential equation on a compact Riemannian manifold, in the sense of Aubin \cite[Chapter 5]{aubin} and Hebey-Vaugon \cite{hebey-vaugon}. Moreover, for $t=2n-1$, it reduces to the classical Yamabe equation \cite{yamabe, lee-parker}
$$ \frac{2(2n-1)}{n-1}\Delta u+\mathrm{scal}(g)\, u = c\,u^{2^*-1} . $$

As a consequence, we can readily adapt classical arguments from the Yamabe problem, for instance the following.
\begin{itemize}
\item The sign of the constant $t$-deformed scalar curvature is a conformal invariant. Indeed, if $\omega$ and $\omega'=u^{\frac{2}{n-1}}\omega$ are conformal metrics with constant $t$-deformed scalar curvature $c=\mu^t(\omega)$ and $c'=\mu^t(\omega')$ respectively, then \eqref{eq:conf-mu-t} gives
$$ \frac{2t}{n-1}\Delta u + cu = c' u^{2^*-1}.$$
Integrating over $X$ with respect to the volume form of $\omega$, we get
$$ c \int_X u \, \frac{\omega^n}{n!} = c' \int_X u^{2^*-1}\,\frac{\omega^n}{n!} , $$
so that $c$ and $c'$ necessarily have the same sign (or are both zero).

\item Moreover, uniqueness up to homothety holds when $c,c'\leq 0$. If $c=c'=0$, this is immediate: the equation above reduces to $\Delta u=0$, which admits only constant solutions. If $c,c'<0$, this follows from the maximum principle. Up to rescaling the metrics, we may assume $c=c'=-1$, so that
$$ \frac{2t}{n-1}\Delta u - u = - u^{2^*-1}.$$
At a minimum point $x_{\rm min}\in X$ of $u$, we get
$$ u(x_{\rm min})^{2^*-1} = -\frac{2t}{n-1}\Delta u(x_{\rm min}) + u(x_{\rm min}) \geq u(x_{\rm min}).$$
Since $2^*>2$, this implies $u(x_{\rm min})\geq 1$. Similarly, evaluating at a maximum point gives $u\leq 1$. Therefore $u \equiv 1$, and $\omega'=\omega$.

\item On compact complex manifolds with non-negative Kodaira dimension, the Gauduchon degree provides an upper bound for the expected volume-normalized constant $t$-deformed scalar curvature, in the regime $0 < t \leq 2n$. Indeed, assume that $\omega=e^f\eta$ has constant $t$-deformed scalar curvature $\mu^t(\omega)=c$, where $\eta$ is the Gauduchon representative of unit volume in the conformal class. The conformal rule \eqref{eq:conf-mu-t}, together with the sign of the torsion term for $0<t\leq 2n$, gives
\begin{eqnarray*}
c\, e^f &=& \mu^t(\eta) + t \Delta f - \frac{t(n-1)}{2} |\nabla f|^2 \\
&=& {\rm scal}^{\rm Ch}(\eta) + \frac{t-2n}{4} |T|^2 + t \Delta f - \frac{t(n-1)}{2} |\nabla f|^2 \\
&\leq& {\rm scal}^{\rm Ch}(\eta) + t \Delta f \,\, ,
\end{eqnarray*}
and integrating over $X$ with respect to the volume form of $\eta$ yields
$$ c\int_X e^{f}\frac{\eta^n}{n!} \leq \Gamma_X (\{\omega\}) \,\, , $$
so that $\Gamma_X(\{\omega\})\leq 0$ (a consequence of the assumption $\mathrm{Kod}(X)\geq 0$) forces $c\leq 0$; the H\"older inequality then refines this to the bound $c\,\mathrm{Vol}(X,g_{\omega})^{\frac{1}{n}}\leq \Gamma_X(\{\omega\})$.
\end{itemize}

\subsection{Euler-Lagrange interpretation}

The main motivation behind the definition of the $t$-deformed scalar curvature is that the corresponding Yamabe-type equation arises as the Euler-Lagrange equation of a natural functional, which we call the $t$-deformed Einstein-Hilbert functional. For a Hermitian metric $\omega$, it is defined as
\begin{equation} \label{eq:EH-t}
\mathrm{EH}^{t}(\omega) := \frac{\int_X\mu^{t}(\omega)\frac{\omega^n}{n!}}{\left(\int_X\frac{\omega^n}{n!}\right)^{\frac{n-1}{n}}} \,\, .
\end{equation}
We consider this functional restricted to a conformal class $\{\omega\}$, parametrized as $u^{\frac{2}{n-1}}\omega$ varying $u \in \mathcal{C}^{\infty}(X,\mathbb{R})$ smooth positive function. On this conformal class, we write
\begin{equation} \label{eq:EH-t-conf}
\mathrm{EH}_{\omega}^{t}(u) := \mathrm{EH}^{t}\left(u^{\frac{2}{n-1}}\omega \right) = \|u\|^{-2}_{L^{2^*}}\int_X\left( u^2\mu^{t}(\omega) +\frac{2t}{n-1}|\nabla u|^2\right)\,\frac{\omega^n}{n!} \,\, ,
\end{equation}
and we omit the reference metric $\omega$ from the notation whenever no confusion can arise.

\begin{theorem} \label{thm:variational}
Let $(X^n,\omega)$ be a compact Hermitian manifold of complex dimension $n$ and fix $t>0$. The $t$-deformed Yamabe equation \eqref{eq:t-DY} is the Euler-Lagrange equation of the $t$-deformed Einstein-Hilbert functional $\mathrm{EH}^t$ restricted to the conformal class of $\omega$.
\end{theorem}

\begin{proof}
Let $u_\varepsilon$ be a smooth path of smooth positive functions on $X$, with $u := u_{\varepsilon=0}$ and $\dot u := \frac{d}{d\varepsilon} \big|_{\varepsilon=0} u_\varepsilon$. A direct computation gives
$$
\frac{d}{d\varepsilon}\Big\vert_{\varepsilon=0} \|u_\varepsilon\|_{L^{2^*}} = \|u\|_{L^{2^*}}^{1-2^*} \int_X u^{2^*-1}\,\dot{u}\,\frac{\omega^n}{n!} \,\, ,
$$
and, writing
$$ \mu^t_{\rm tot}(u) := \int_X\left( u^2\mu^{t}(\omega)+\frac{2t}{n-1}|\nabla u|^2\right)\frac{\omega^n}{n!} $$
for the numerator in \eqref{eq:EH-t-conf},
$$
\frac{d}{d\varepsilon}\Big\vert_{\varepsilon=0} \mu^t_{\rm tot}(u_\varepsilon) = 2\int_X \left(\frac{2t}{n-1}\Delta u +\mu^t(\omega)\,u\right)\dot{u}\,\frac{\omega^n}{n!} \,\, .
$$
Combining these, the first variation of $\mathrm{EH}^t_\omega(u) = \|u\|_{L^{2^*}}^{-2}\mu^t_{\rm tot}(u)$ is
$$
\frac{d}{d\varepsilon}\Big\vert_{\varepsilon=0}\mathrm{EH}^t (u_\varepsilon) = \frac{2}{\|u\|_{L^{2^*}}^2}\int_X \left(\frac{2t}{n-1}\Delta u+\mu^t(\omega)\,u -\frac{\mu^t_{\rm tot}(u)}{\|u\|_{L^{2^*}}^{2^*}}\,u^{2^*-1}\right)\dot{u}\,\frac{\omega^n}{n!} \,\, .
$$
Hence $\frac{d}{d\varepsilon}\big\vert_{\varepsilon=0}\mathrm{EH}^t (u_\varepsilon) = 0$ for every $\dot{u}$ if and only if $u$ satisfies \eqref{eq:t-DY} with $c=\frac{\mu^t_{\rm tot}(u)}{\|u\|_{L^{2^*}}^{2^*}}$.
\end{proof}

\subsection{Deformed Yamabe invariants}

The functional $\mathrm{EH}^t$ in \eqref{eq:EH-t-conf} is bounded below. Indeed, since $t>0$, the gradient term is non-negative, and the H\"older inequality gives
\begin{eqnarray*}
\mathrm{EH}^t(u)
&=& \frac{\int_X \left(u^2\mu^t(\omega)+\frac{2t}{n-1}|\nabla u|^2\right)\,\frac{\omega^n}{n!}}{\|u\|_{L^{2^*}}^2} \geq \frac{\int_X u^2\mu^t(\omega)\,\frac{\omega^n}{n!}}{\|u\|_{L^{2^*}}^2} \\
&\geq& - \|\mu^t(\omega)\|_{L^k}\cdot\frac{\|u^2\|_{L^{2^*/2}}}{\|u\|_{L^{2^*}}^2} = - \|\mu^t(\omega)\|_{L^k} ,
\end{eqnarray*}
where $k = (1-2/2^*)^{-1}$ is the conjugate exponent of $2^*/2$.

This allows us to define the $t$-deformed Yamabe invariant of $\{\omega\}$ as
\begin{equation}\label{eq:Yam_inv}
\lambda_{X}^{t}(\{\omega\}) := \inf \left\{\mathrm{EH}^{t}(\omega') : \omega' \in \{\omega\}\right\} \in \mathbb R .
\end{equation}

\begin{remark}
In the following section, we show that $\lambda^t_X(\{\omega\})$ is bounded above by a universal constant depending only on the complex dimension $n$ and the parameter $t$, independent of $\{\omega\}$. More precisely,
$$ \lambda^t_X(\{\omega\}) \leq t \Lambda_n , \quad \text{ where } \Lambda_n:= 2n\,{\rm Vol}(S^{2n})^\frac{1}{n}, $$
see Theorem~\ref{thm:deformed-yamabe-existence} for more details.
\end{remark}

\begin{remark}\label{rmk:EH-p}
More generally, the same lower bound holds for the functional
\begin{equation}\label{eq:EH-t-p}
\mathrm{EH}^t_p(u) := \frac{\mu^t_{\mathrm{tot}}(u)}{\|u\|_{L^p}^2}
\end{equation}
for any $2\leq p\leq+\infty$, where the $L^{2^*}$ norm in the denominator is replaced by $L^p$.
Similar computations give
$$ \mathrm{EH}^t_p(u) \geq - \|\mu^t(\omega)\|_{L^k}\cdot\frac{\|u^2\|_{L^{p/2}}}{\|u\|_{L^p}^2} = - \|\mu^t(\omega)\|_{L^k} , $$
where $k=(1-2/p)^{-1}=\frac{p}{p-2}$ is the conjugate exponent of $p/2$.
Here is where the value $2^*$ plays a crucial role: for $2 \leq p < 2^*$, the functional $\mathrm{EH}^t_p$ admits a minimizer. To see this, take a minimizing sequence $\{u_j\}$ for $\mathrm{EH}^t_p$ among positive smooth functions; by homogeneity of the functional, we may assume $\|u_j\|_{L^p}=1$. We claim that $\{u_j\}$ is bounded in $W^{1,2}$. Indeed,
\begin{eqnarray*}
\frac{2t}{n-1}\,\|u\|_{W^{1,2}}^2
&=& \frac{2t}{n-1}\,\int_X (u^2+|\nabla u|^2)\,\frac{\omega^n}{n!} \\
&=& \mathrm{EH}^t_p(u) + \int_X \left(\frac{2t}{n-1}-\mu^t(\omega)\right)u^2\,\frac{\omega^n}{n!} \\
&\leq& \mathrm{EH}^t_p(u) + \left\| \frac{2t}{n-1}-\mu^t(\omega) \right\|_{L^k} ,
\end{eqnarray*}
where $k=(1-2/p)^{-1}$ as above.
For $p<2^*$, the Sobolev embedding $W^{1,2}\hookrightarrow L^p$ is compact. Therefore, up to passing to a subsequence, $u_j$ converges weakly in $W^{1,2}$ and strongly in $L^p$ to some $\tilde{u}\in W^{1,2}$, which minimizes $\mathrm{EH}^t_p$ and satisfies $\|\tilde{u}\|_{L^p}=1$.
By adapting Theorem~\ref{thm:variational} with $2^*$ replaced by $p$, we get that $\tilde{u}$ is a weak solution of the semilinear equation
\begin{equation*}
\frac{2t}{n-1}\,\Delta u + \mu^t(\omega)\, u = \lambda\,u^{p-1} ,
\end{equation*}
for $\lambda = \inf \mathrm{EH}^t_p$. By elliptic regularity, $\tilde{u}$ is smooth. Moreover, $\tilde{u}\geq 0$ as a limit of positive functions, and $\tilde{u}\not\equiv 0$ since $\|\tilde{u}\|_{L^p}=1$.
In fact, the maximum principle implies that $\tilde{u}>0$. Indeed, the equation can be rewritten as $\frac{2t}{n-1}\,\Delta\tilde{u} + \left(\mu^t(\omega) - \lambda\,\tilde{u}^{p-2}\right)\tilde{u} = 0$, so that $\tilde{u}\geq 0$ satisfies $\frac{2t}{n-1}\,\Delta\tilde{u} \geq -C\tilde{u}$ for some constant $C>0$. By the strong maximum principle, either $\tilde{u}\equiv 0$ or $\tilde{u}>0$ everywhere.
\end{remark}

\begin{remark}
We collect here some results comparing the deformed Yamabe invariants $\lambda^t_X(\{\omega\})$ with the Gauduchon degree $\Gamma_X(\{\omega\})$ and with the classical Yamabe invariant $\lambda_X^{\rm Yam}(g)$.
\begin{itemize}
\item For $0 < t \leq 2n$, we have
$$ \lambda^{t}_{X}(\{\omega\}) \leq \Gamma_{X}(\{\omega\}) . $$
Indeed, taking the Gauduchon metric $\eta$ of unit volume in $\{\omega\}$ as a test metric, we get
\begin{eqnarray*}
\lambda^t_X(\{\omega\})
&\leq& \mathrm{EH}^t(\eta) = \int_X \mu^t(\eta) \frac{\eta^n}{n!} \\
&=& \int_X \mathrm{scal}^{\mathrm{Ch}}(\eta) \, \frac{\eta^n}{n!} + \frac{t-2n}{4} \int_X |T(\eta)|^2 \, \frac{\eta^n}{n!} \\
&\leq& \int_X \mathrm{scal}^{\mathrm{Ch}}(\eta) \, \frac{\eta^n}{n!} = \Gamma_X(\{\omega\}) .
\end{eqnarray*}
(Note that, for $t=2n$, i.e., in the momentum case, equality $\lambda^{2n}_{X}(\{\omega\}) = \Gamma_{X}(\{\omega\})$ implies that $\eta$ realizes the infimum of ${\rm EH}^{2n}$ and therefore $\mathrm{scal}^{\mathrm{Ch}}(\eta)$ is constant. In the non-positive case, the converse also holds, by uniqueness up to homothety.)
\item When the conformal class contains a K\"ahler metric $\omega_K\in\{\omega\}$, then
\begin{eqnarray*}
\lambda^t_X(\{\omega\}) \leq \lambda^{\rm Yam}_X(\{g\}) & \text{ for } & 0 < t < 2n-1 , \\
\lambda^t_X(\{\omega\}) \geq \lambda^{\rm Yam}_X(\{g\}) & \text{ for } & t\geq 2n-1 .
\end{eqnarray*}
This follows directly by comparing with the classical Einstein-Hilbert functional
$$ {\rm EH}^{\rm Yam}(u) = \frac{ \int_X \left(u^2\,{\rm scal}(g_K) + \frac{2(2n-1)}{n-1}|\nabla u|^2\right)\,\frac{\omega^n}{n!} }{ \left( \int_X u^{2^*}\, \frac{\omega^n}{n!} \right)^{\frac{n-1}{n}} } $$
with respect to $g_K$, for which $\mu^t(\omega_K)={\rm scal}(g_k)$.
\item In particular, if the conformal class $\{\omega\}$ contains a Kähler metric $\omega_K$ that minimizes the classical Einstein-Hilbert functional, then
$$ \lambda^{t}_{X}(\{\omega\}) = \lambda^{\mathrm{Yam}}_X(\{g\}) \text{ for } 2n-1 \leq t \leq 2n . $$
Indeed, up to normalizing the volume, $\lambda^{\mathrm{Yam}}_X(\{\omega\})=\int_X \mathrm{scal}(g_K)\,\frac{\omega_K^n}{n!} = \Gamma_X(\{\omega\})$, and the conclusion follows by combining the previous two items.
\item In particular, if the conformal class $\{\omega\}$ contains a K\"ahler-Einstein metric $\omega_{\mathrm{KE}}$, then
$$ \lambda^t_X(\{\omega\}) = \mathrm{scal}(g_{\mathrm{KE}})\cdot\mathrm{Vol}(X,\omega_{\mathrm{KE}})^{\frac{1}{n}} \text{ for } 2n-1 \leq t \leq 2n . $$
For example, for $X=\mathbb CP^n$, endowed with the Fubini-Study metric $\omega_{\mathrm{FS}}$, we get
$$ \lambda^t_{\mathbb CP^n}(\{\omega_{\mathrm{FS}}\}) = \frac{4n(n+1)\,\pi}{(n!)^{\frac{1}{n}}} \qquad \text{for } 2n-1 \leq t \leq 2n . $$
Note indeed that the quantity $\mathrm{scal}(g)\cdot\mathrm{Vol}(X,g)^{\frac{1}{n}}$ is invariant under homotheties, so the value above does not depend on the normalization of $\omega_{\mathrm{FS}}$; for definiteness, it can be computed with the normalization in which the holomorphic sectional curvature equals $4$, so that $\mathrm{Ric}(g_{\mathrm{FS}})=2(n+1)\,g_{\mathrm{FS}}$, whence $\mathrm{scal}(g_{\mathrm{FS}})=4n(n+1)$ and $\mathrm{Vol}(\mathbb CP^n,\omega_{\mathrm{FS}})=\frac{\pi^n}{n!}$. In particular, for $n=1$, we have $\lambda^t_{\mathbb CP^1}(\{\omega_{\rm FS}\})=8\pi$, for $1\leq t \leq 2$.
\end{itemize}
\end{remark}

\subsection{Existence results}

In this section, we prove that the $t$-deformed Yamabe invariant is bounded above by a universal constant depending only on the complex dimension $n$ and the parameter $t$, and then we provide a pointwise curvature condition ensuring the existence of solutions to the $t$-deformed Yamabe equation. The argument, following \cite{aubin-1976}, is based on evaluating the deformed Einstein-Hilbert functional on specific test functions arising from the classical Sobolev inequality (Theorem~\ref{thm:talenti}) and a canonical choice of holomorphic normal coordinates (Theorem~\ref{thm:holomorphic-normal}). We start by recalling these two results.

\begin{theorem}[{\cite{aubin,talenti}}]\label{thm:talenti}
Let $m\geq 3$ and $2^*=\frac{2m}{m-2}$. The Sobolev embedding
$$ W^{1,2}(\mathbb{R}^m)\hookrightarrow L^{2^*}(\mathbb{R}^m) $$
is continuous, and the sharp inequality
$$ \|u\|^2_{L^{2^*}(\mathbb{R}^m)} \leq \sigma_m \|\nabla u\|^2_{L^2(\mathbb{R}^m)} $$
holds for all $u \in W^{1,2}(\mathbb{R}^m)$, with optimal constant
$$ \sigma_m = \frac{4}{m(m-2)\mathrm{Vol}(S^m)^{\frac{2}{m}}} , $$
where $\mathrm{Vol}(S^m)$ is the volume of the unit $m$-sphere with respect to the standard metric.
Equality holds if and only if $u$ is a constant multiple or translate of
\begin{equation}\label{eq:sobolev-minimizers}
u_\varepsilon(x) = \left( \frac{\varepsilon}{\varepsilon^2 + |x|^2} \right)^{\frac{m-2}{2}}
\end{equation}
for some $\varepsilon>0$.
\end{theorem}

\begin{remark}
For completeness, we recall that
$$ \mathrm{Vol}(S^m) = \frac{2\pi^{\frac{m+1}{2}}}{\Gamma\!\left(\frac{m+1}{2}\right)} , $$
although we will not need the explicit value in what follows. In particular, for $m=2n$, using $\Gamma\!\left(n+\frac{1}{2}\right)=\frac{(2n)!}{4^n\,n!}\sqrt{\pi}$, we get
$$ \mathrm{Vol}(S^{2n}) = \frac{2^{2n+1}\,n!}{(2n)!}\,\pi^n . $$
Therefore,
$$ \sigma_{2n} = \frac{1}{n(n-1)\,\mathrm{Vol}(S^{2n})^{\frac{1}{n}}} = \frac{1}{n(n-1)}\cdot\frac{((2n)!)^{\frac{1}{n}}}{4\cdot 2^{\frac{1}{n}}\,(n!)^{\frac{1}{n}}\,\pi} , $$
and we set
$$ \Lambda_n := \frac{2}{n-1}\cdot\frac{1}{\sigma_{2n}} = 2n\,\mathrm{Vol}(S^{2n})^{\frac{1}{n}} . $$

For small values of $n$, we have
\begin{center}
\renewcommand{\arraystretch}{1.6}
\begin{tabular}{@{}c | c | c@{}}
\toprule
$n$ & $\mathrm{Vol}(S^{2n})$ & $\Lambda_n$ \\
\midrule
$1$ & $4\,\pi$                                  & $8\,\pi \approx 25.133$ \\
$2$ & $\dfrac{8}{3}\,\pi^{2}$                    & $8\sqrt{\dfrac{2}{3}}\,\pi \approx 20.521$ \\
$3$ & $\dfrac{16}{15}\,\pi^{3}$                  & $12\left(\dfrac{2}{15}\right)^{\frac{1}{3}}\pi \approx 19.259$ \\
$4$ & $\dfrac{32}{105}\,\pi^{4}$                 & $16\left(\dfrac{2}{105}\right)^{\frac{1}{4}}\pi \approx 18.674$ \\
$5$ & $\dfrac{64}{945}\,\pi^{5}$                 & $20\left(\dfrac{2}{945}\right)^{\frac{1}{5}}\pi \approx 18.336$ \\
$6$ & $\dfrac{128}{10395}\,\pi^{6}$               & $24\left(\dfrac{2}{10395}\right)^{\frac{1}{6}}\pi \approx 18.116$ \\
\bottomrule
\end{tabular}
\end{center}

One can show that the sequence $\Lambda_n$ is strictly decreasing. For large values of $n$, we have:
\begin{center}
\renewcommand{\arraystretch}{1.4}
\begin{tabular}{@{}c | c@{}}
\toprule
$n$ & $\Lambda_n$ \\
\midrule
$10$    & $\approx 17.689$ \\
$20$    & $\approx 17.380$ \\
$30$    & $\approx 17.279$ \\
$50$    & $\approx 17.199$ \\
$100$   & $\approx 17.139$ \\
$500$   & $\approx 17.091$ \\
$1000$  & $\approx 17.085$ \\
$10000$ & $\approx 17.080$ \\
\bottomrule
\end{tabular}
\end{center}
By Stirling's formula, $\Gamma(z)\sim\sqrt{2\pi/z}\,(z/e)^z$ as $z\to+\infty$; applying this with $z=n+\frac{1}{2}$ gives
\begin{eqnarray*}
\mathrm{Vol}(S^{2n})
&\sim& \sqrt{\frac{2}{\pi}}\cdot\sqrt{n+\frac{1}{2}}\cdot\pi^{n+\frac{1}{2}}\cdot\left(\frac{e}{n+\frac{1}{2}}\right)^{n+\frac{1}{2}} \\
&=& \sqrt{2e} \cdot\left(\frac{\pi e}{n+\frac{1}{2}}\right)^{n}
\qquad \text{ as } n\to+\infty . 
\end{eqnarray*}
Taking the $n$-th root and letting $n\to+\infty$, one finds $\mathrm{Vol}(S^{2n})^{1/n}\sim\pi e/n$, whence
$$ \Lambda_n = 2n\,\mathrm{Vol}(S^{2n})^{1/n} \longrightarrow 2\pi e \qquad \text{ as }n\to+\infty , $$
where $2\pi e\approx17.0795$ is related to the entropy of the standard Gaussian.
\end{remark}

The second result we need concerns the existence of special coordinates in the
non-K\"ahler setting (see also \cite[Lemma 3.4]{liu-yang}), extending the classical
construction available in the K\"ahler case, where $d\omega=0$ guarantees
$\partial_i h_{j\bar\ell}(0)=\partial_j h_{i\bar\ell}(0)$
(see, e.g., \cite[pp.~107-108]{griffiths-harris}).

\begin{theorem}[holomorphic normal coordinates]\label{thm:holomorphic-normal}
Let $X$ be a complex manifold endowed with a Hermitian metric $\omega$. For every point $p \in X$, there exist local holomorphic coordinates centred at $p$ such that the local components $h_{i\bar{j}}$ of $\omega$ satisfy
\begin{equation} \label{eq:normcoord}
h_{i\bar{j}}(0) = \delta_{i}^{j} , \quad
\frac{\partial h_{j\bar{\ell}}}{\partial z^i} (0) + \frac{\partial h_{i\bar{\ell}}}{\partial z^j} (0) = 0 .
\end{equation}
\end{theorem}

\begin{proof}
Let $\{z^i\}$ be holomorphic coordinates centred at $p$ with $h_{i\bar{j}}(0)=\delta_{i}^{j}$, and set
$$
w^k := z^k+\frac{1}{4}\left(\frac{\partial h_{s\bar{k}}}{\partial z^r} (0) + \frac{\partial h_{r\bar{k}}}{\partial z^s} (0)\right)z^rz^s .
$$
By the holomorphic inverse function theorem, $\{w^k\}$ is a system of local holomorphic coordinates centred at $p$ (possibly after shrinking the neighbourhood). Denoting by $\tilde{h}_{i\bar{j}}$ the metric components in the new coordinates, the change of variables gives
$$
z^k = w^k -\frac{1}{4}\left(\frac{\partial h_{s\bar{k}}}{\partial z^r} (0) + \frac{\partial h_{r\bar{k}}}{\partial z^s} (0)\right)w^rw^s +O(|w|^3)
$$
and
$$ \tilde{h}_{i\bar{j}} = h_{k\bar{\ell}}\frac{\partial z^k}{\partial w^i}\frac{\partial \bar{z}^{\ell}}{\partial \bar{w}^{j}} . $$
A direct computation then gives $\tilde{h}_{i\bar{j}}(0) = \delta_{i}^{j}$ and
\begin{eqnarray*}
\frac{\partial \tilde{h}_{j\bar{\ell}}}{\partial w^i}(0) +\frac{\partial \tilde{h}_{i\bar{\ell}}}{\partial w^j}(0)
&=& \frac{\partial h_{j\bar{\ell}}}{\partial z^i}(0)
+\frac{\partial h_{i\bar{\ell}}}{\partial z^j}(0)
+2\frac{\partial^2 z^{\ell}}{\partial w^i\partial w^j}(0) \\
&=& \frac{\partial h_{j\bar{\ell}}}{\partial z^i}(0)
+\frac{\partial h_{i\bar{\ell}}}{\partial z^j}(0)
+4 \left( -\frac{1}{4} \frac{\partial h_{j\bar\ell}}{\partial z^j}(0)+\frac{\partial h_{i\bar \ell}}{\partial z^i}(0) \right) = 0 ,
\end{eqnarray*}
which is precisely \eqref{eq:normcoord}.
\end{proof}

We are now ready to prove the following existence result, due to \cite[Théorème 1]{aubin-1976} and \cite[Theorem 2.1]{hebey-ictp}.

\begin{theorem}[{\cite{aubin-1976,hebey-ictp}}]\label{thm:deformed-yamabe-existence}
Let $X$ be a compact complex manifold of complex dimension $n\geq3$, endowed with a Hermitian metric $\omega$, and fix $t>0$. Then
$$ \lambda^t_X(\{\omega\}) \leq t\Lambda_n . $$
Moreover, if the strict inequality
$$ \lambda^t_X(\{\omega\}) < t\Lambda_n $$
holds, then there exists a positive smooth solution to the $t$-deformed Yamabe equation \eqref{eq:t-DY}, which is a minimizer of the $t$-deformed Einstein-Hilbert functional \eqref{eq:EH-t-conf} restricted to the conformal class $\{\omega\}$.
Finally, if there exists a point $p\in X$ such that
$$ (t-2n+1)\,\mu^0(\omega)(p) > 0 , \quad \text{where } \mu^0(\omega) = \mathrm{scal}^{\mathrm{Ch}}(\omega) - \frac{n}{n-1} d^*\theta - \frac{n}{2}\,|T|^2 , $$
then $\lambda^t_X(\{\omega\}) < t\Lambda_n$, and the previous conclusion applies.
\end{theorem}

\begin{proof}
We split the proof into several steps, following the argument in \cite{aubin-1976} as described in \cite{angella-pediconi-scarpa-spotti-windare}.

\smallskip
\noindent{\itshape Step 0. Setup and notation.}
Choose a point $p\in X$ and fix holomorphic normal coordinates $\{z^i\}$ centred at $p$, as provided by Theorem~\ref{thm:holomorphic-normal}. Write $z^i=x^{2i-1}+\sqrt{-1}\,x^{2i}$ and let $dx=dx^1\wedge\cdots\wedge dx^{2n}$ denote the standard Lebesgue measure on $\mathbb{R}^{2n}$, so that $\frac{\omega^n}{n!}=2^n\det(h_{i\bar{j}})\,dx$.
For ease of notation, we set $\nu_{2n-1}:=\mathrm{Vol}(S^{2n-1})$.

For later use, we set, for $k<n$,
$$
I_k := \int_0^{+\infty} \frac{r^{2n-1+2k}}{(1+r^2)^{2n}} \,dr , $$
and, for $k<n-2$,
$$ J_k := \int_0^{+\infty} \frac{r^{2n-1+2k}}{(1+r^2)^{2n-2}} \,dr $$
(see, e.g., \cite[Lemma 3.5]{lee-parker}). We will use the identities
$$
\frac{I_1}{I_0} = \frac{n}{n-1} , \qquad \frac{J_0}{I_0} = \frac{2(2n-1)}{n-2} , \qquad \frac{I_2}{I_0} = \frac{n(n+1)}{(n-1)(n-2)} ,
$$
as well as the relation
\begin{equation}\label{eq:sigma-I0}
\frac{1}{\sigma_{2n}} = 4n(n-1)(\nu_{2n-1}\,I_0)^{1/n} .
\end{equation}

\smallskip
\noindent{\itshape Step 1. Test functions.}
For $\varepsilon>0$, let $U_\varepsilon$ denote the standard bubble
$$ U_\varepsilon(z) := \left(\frac{\varepsilon}{\varepsilon^2+|z|^2}\right)^{n-1} $$
as in Theorem~\ref{thm:talenti}, which satisfies
$$ \Lambda_n\,\|U_\varepsilon\|^2_{L^{2^*}(\mathbb{R}^{2n})} = \frac{2}{n-1}\,\|\nabla U_\varepsilon\|^2_{L^2(\mathbb{R}^{2n})} . $$
Fix $\delta>0$ small enough so that the ball $B(0,\delta)$ is contained in the local chart, and set $r:=|z|$ the radial coordinate. Let $\eta$ be a smooth radial cutoff function satisfying
$$ \left\{ \begin{array}{l}
0\leq\eta\leq1 ,\\
\eta(z)=1 \text{ for } z \in B_{\delta/2},\\
\mathrm{supp}(\eta)\subset B_{\delta} ,
\end{array}\right. $$
and consider the test function
$$
u_\varepsilon(z) := \eta(|z|)\, U_\varepsilon(z) .
$$
The goal is to evaluate
$$ \mathrm{EH}^t(u_\varepsilon) = \frac{\int_X \left( u_\varepsilon^2\,\mu^t(\omega) + \frac{2t}{n-1}|\nabla u_\varepsilon|^2\right)\,\frac{\omega^n}{n!}}{\left(\int_X u_\varepsilon^{\frac{2n}{n-1}}\,\frac{\omega^n}{n!}\right)^{\frac{n-1}{n}}} $$
and to study its behaviour as $\varepsilon\to 0$.

\smallskip
\noindent{\itshape Step 2. Taylor expansion of the metric tensor.}
In holomorphic normal coordinates at $p$, the metric coefficients admit the Taylor expansion
$$ h_{i\bar{j}}(z) = \delta_{i}^{j} + A_{i\bar{j}}(z) , $$
where
\begin{eqnarray*}
A_{i\bar{j}}(z)
&:=& \frac{1}{2} T_{ki}^{j}(0) z^k + \frac{1}{2} \overline{T_{\ell j}^{i}(0)} \,\bar{z}^\ell \\
&& - \left(\Theta_{k\bar{\ell} i\bar{j}}(0) - \frac{1}{4} T_{ki}^{q}(0)\overline{T_{\ell j}^{q}(0)} \right) z^k\bar{z}^\ell \\
&& + \mathrm{(ch.t.)} + o(|z|^2)
\end{eqnarray*}
is $O(|z|)$.
Here and after, ``$\mathrm{(ch.t.)}$'' denotes charged terms, of order $2$ in this case, i.e., terms of bidegree $(2,0)$ and $(0,2)$.

For the inverse of the metric, we have
\begin{eqnarray*}
h^{\bar{j}i}(z)
&=& \delta_{j}^{i} - A_{j\bar{i}}(z) + A_{j\bar{p}}(z)A_{p\bar{i}}(z) + o(|z|^2) \\
&=& \delta_{j}^{i} -\frac{1}{2}T_{kj}^{i}(0)z^k -\frac{1}{2}\overline{T_{\ell i}^{j}(0)}\,\bar{z}^{\ell} \\
&& +\left(\Theta_{k\bar{\ell}j\bar{i}}(0) +\frac{1}{4}T_{kp}^{i}(0)\overline{T_{\ell p}^{j}(0)} \right)z^k\bar{z}^{\ell} \\
&& +\mathrm{(ch.t.)} +o(|z|^2) .
\end{eqnarray*}

Finally, the Taylor expansion of the determinant of the metric is
\begin{eqnarray*}
\det(h_{m\bar{s}})(z)
&=& 1 + A_{i\bar{i}}(z) + \frac{1}{2} \left( A_{i\bar{i}}(z)^2 - A_{i\bar{j}}(z)A_{j\bar{i}}(z) \right) + o(|z|^2) \\
&=& 1 + \frac{1}{2} \theta_k(0) \, z^k
+ \frac{1}{2} \overline{\theta_{\ell}(0)} \, \bar{z}^\ell \\
&& - \left(\Theta_{k\bar{\ell} p\bar{p}}(0) - \frac{1}{4} \theta_k(0)\overline{\theta_{\ell}(0)} \right) z^k \bar{z}^\ell \\
&& +\mathrm{(ch.t.)} + o(|z|^2) .
\end{eqnarray*}

\smallskip
\noindent{\itshape Step 3. Angular averages of the volume form.}
Since we are working with radial test functions, some terms do not contribute to the integrals. Let $S^{2n-1}_r \subset \mathbb{C}^n$ denote the sphere of radius $r$ centred at the origin, and let $d\sigma_r$ denote the area element of $S^{2n-1}_r$ induced by the flat metric of $\mathbb{C}^n$. The relevant angular averages are
\begin{eqnarray*}
\frac{1}{\nu_{2n-1}}\int_{S_r^{2n-1}} d\sigma_r &=& r^{2n-1} , \\
\frac{1}{\nu_{2n-1}}\int_{S_r^{2n-1}} z^j\bar{z}^i\,d\sigma_r &=& \frac{r^{2n+1}}{n}\,\delta_{i}^{j} , \\
\frac{1}{\nu_{2n-1}}\int_{S_r^{2n-1}} z^j\bar{z}^iz^k\bar{z}^{\ell}\,d\sigma_r &=& \frac{r^{2n+3}}{n(n+1)}\left(\delta_{i}^{j}\delta_{\ell}^{k}+\delta_{\ell}^{j}\delta_{i}^{k}\right) ,
\end{eqnarray*}
while the angular average of charged terms vanishes.

We compute the Taylor expansion of the angular average of the volume form as $r\to 0$:
\begin{eqnarray*}
\lefteqn{ \frac{1}{\nu_{2n-1}}\int_{S_r^{2n-1}} \det(h_{m\bar{s}}) \,d\sigma_r } \\
&=& r^{2n-1} -\frac{1}{n}\left( \Theta_{i\bar{i} j \bar{j}}(0)-\frac{1}{4}\theta_i(0)\overline{\theta_i(0)} \right) r^{2n+1} +o(r^{2n+1}) \\
&=& r^{2n-1} + \Phi(p)\, r^{2n+1} + o(r^{2n+1}) ,
\end{eqnarray*}
where
$$ \Phi := -\frac{1}{2n}\left(\mathrm{scal}^{\mathrm{Ch}} -\frac{1}{4}|\theta|^2\right) . $$

We also need the following computation:
\begin{eqnarray*}
\lefteqn{ h^{\bar{j}i}(z)z^j\bar{z}^i\det(h_{m\bar{s}}(z)) } \\
&=& \delta_{j}^{i}z^j\bar{z}^i +\left(-\Theta_{k\bar{\ell}p\bar{p}}(0)\delta_{j}^{i} +\frac{1}{4}\theta_k(0)\overline{\theta_{\ell}(0)}\delta_{j}^{i} +\Theta_{k\bar{\ell}j\bar{i}}(0) +\frac{1}{4}T_{kp}^{i}(0)\overline{T_{\ell p}^{j}(0)} \right. \\
&& \left. -\frac{1}{4}T_{kj}^{i}(0)\overline{\theta_{\ell}(0)} -\frac{1}{4}\theta_k(0)\overline{T_{\ell i}^{j}(0)} \right)z^j\bar{z}^iz^k\bar{z}^{\ell} +\mathrm{(ch.t.)} +o(|z|^4) ,
\end{eqnarray*}
where $\mathrm{(ch.t.)}$ denotes charged terms of order $\leq 4$. Therefore,
\begin{eqnarray*}
\lefteqn{ \frac{1}{\nu_{2n-1}}\int_{S_r^{2n-1}} h^{\bar{j}i}(z)z^j\bar{z}^i\det(h_{m\bar{s}})\,d\sigma_r } \\
&=& r^{2n+1} +\frac{1}{n(n+1)}\left(-n\Theta_{i\bar{i}j\bar{j}}(0) +\Theta_{i\bar{j}j\bar{i}}(0) +\frac{n+2}{4}\theta_i(0)\overline{\theta_i(0)} +\frac{1}{4}T_{ij}^{k}(0)\overline{T_{ij}^{k}(0)}\right)r^{2n+3} +o(r^{2n+3}) \\
&=& r^{2n+1} + \Psi(p)\, r^{2n+3} + o(r^{2n+3}) ,
\end{eqnarray*}
where
\begin{eqnarray*}\label{tilde-scal}
\Psi &:=& \frac{1}{2n(n+1)}\left(-n\,\mathrm{scal}^{\mathrm{Ch}}(p) +\widetilde{\mathrm{scal}}{}^{\mathrm{Ch}}(p) +\frac{n+2}{4}|\theta(p)|^2 +\frac{1}{4}|T(p)|^2\right) \\
&=& -\frac{1}{2n(n+1)}\left((n-1)\,\mathrm{scal}^{\mathrm{Ch}} +d^*\theta -\frac{n-2}{4}|\theta|^2 -\frac{1}{4}|T|^2\right) .
\end{eqnarray*}

\smallskip
\noindent{\itshape Step 4. Taylor expansion of the denominator.}
We compute
\begin{eqnarray*}
\|u_\varepsilon\|_{L^{2^*}}^{2^*}
&=& 2^n\int_{B(0,\delta)} \eta(|z|)^{2^*} U_\varepsilon(z)^{2^*} \det(h_{m\bar{s}})(z) \, dx \\
&=& 2^n\nu_{2n-1}\int_{0}^{\delta/\varepsilon} \frac{1}{\varepsilon^{2n}}\eta(r\varepsilon)^{2^*} \frac{r^{2n-1}}{(1+r^2)^{2n}} \cdot \left( (r\varepsilon)^{2n-1} +\Phi(p) (r\varepsilon)^{2n+1} + o(\varepsilon^{2n+1})\right) \varepsilon \,dr \\
&=& 2^n\nu_{2n-1} \left( \int_{0}^{+\infty} \frac{r^{2n-1}}{(1+r^2)^{2n}} \, dr + O(\varepsilon^{2n}) \right) \\
&& + 2^n \nu_{2n-1} \Phi(p) \left(\int_{0}^{+\infty} \frac{r^{2n+1}}{(1+r^2)^{2n}}\,dr + O(\varepsilon^{2n-2}) \right) \varepsilon^2 + o(\varepsilon^2) \\
&=& 2^n\nu_{2n-1} \left( I_0 + \Phi(p) I_1\varepsilon^2 + o(\varepsilon^2) \right) .
\end{eqnarray*}
Raising to the power $\frac{2}{2^*}=\frac{n-1}{n}$, we obtain
$$ \|u_\varepsilon\|_{L^{2^*}}^{2} = 2^{n-1}\left(\nu_{2n-1}I_0\right)^{\frac{n-1}{n}} \left( 1 + \Phi(p)\,\varepsilon^2 + o(\varepsilon^2) \right) . $$

\smallskip
\noindent{\itshape Step 5. Taylor expansion of the curvature term.}
We compute
\begin{eqnarray*}
\lefteqn{ \int_X u_\varepsilon^2 \mu^t(\omega) \,\frac{\omega^n}{n!} } \\
&=& 2^n\int_{B(0,\delta)} \eta(|z|)^2 U_\varepsilon(z)^2 \mu^t(\omega)(z) \det(h_{m\bar{s}})(z) \, dx \\
&=& 2^n\nu_{2n-1}\varepsilon^{3-2n} \int_{0}^{\delta/\varepsilon} \eta(r\varepsilon)^2 \frac{1}{(1+r^2)^{2n-2}} \left( \mu^t(\omega)(p) r^{2n-1} \varepsilon^{2n-1} + o(\varepsilon^{2n-1}) \right) \,dr \\
&=& 2^n\nu_{2n-1}\varepsilon^{3-2n} \left( \mu^t(\omega)(p)\varepsilon^{2n-1} \int_{0}^{+\infty} \frac{r^{2n-1}}{(1+r^2)^{2n-2}} \,dr + O(\varepsilon^{2n-4}) + o(\varepsilon^{2n-1}) \right) \\
&=& 2^n\nu_{2n-1}\mu^t(\omega)(p)\, J_0 \,\varepsilon^2 + o(\varepsilon^2) .
\end{eqnarray*}
Note that, for $n=2$, the integral $J_0$ diverges logarithmically, and the term $\mu^t(\omega)(p)\,J_0\,\varepsilon^2$ here is replaced by a term of order $\varepsilon^2\log(1/\varepsilon)$.

\smallskip
\noindent{\itshape Step 6. Taylor expansion of the gradient term.} We compute
\begin{eqnarray*}
\lefteqn{ \|\nabla u_\varepsilon \|_{L^2}^2 } \\
&=& 2^{n}\int_{B(0,\delta)} 2h^{\bar{\ell} k} \left(u_\varepsilon'(z) \frac{\bar{z}^k}{2|z|}\right) \left(u_\varepsilon'(z) \frac{z^\ell}{2|z|}\right) \det(h_{m\bar{s}})(z) \, dx \\
&=& 2^{n-1}\nu_{2n-1}\varepsilon^2 \int_{0}^{\delta/\varepsilon} r^{2n-1} \left(\frac{\eta'(r\varepsilon)}{(1+r^2)^{n-1}} - \frac{2(n-1)}{\varepsilon} \frac{r\,\eta(r\varepsilon)}{(1+r^2)^n} \right)^2 (1+\Psi(p)r^2\varepsilon^2 +o(\varepsilon^2)) \, dr \\
&=& 2^{n+1}(n-1)^2\nu_{2n-1} \int_{0}^{+\infty} \left(\frac{r^{2n-1}}{(1+r^2)^{2n}}+O(\varepsilon^{2n-2}) \right) (r^2+\Psi(p) r^4\varepsilon^2+o(\varepsilon^2)) \, dr \\
&=& 2^{n+1}(n-1)^2\nu_{2n-1} \left(I_1+\Psi(p)\, I_2\,\varepsilon^2 + o(\varepsilon^2)\right) .
\end{eqnarray*}

\smallskip
\noindent{\itshape Step 7. Taylor expansion of the numerator.}
Combining Steps 5 and 6, we obtain
\begin{eqnarray*}
\lefteqn{ \int_X \left(u_{\varepsilon}^2\,\mu^t(\omega) + \frac{2t}{n-1}|\nabla u_{\varepsilon}|^2\right) \frac{\omega^n}{n!} } \\
&=& 2^n\nu_{2n-1} \left( \mu^t(\omega)(p)\, J_0 \,\varepsilon^2 + \frac{2t}{n-1} \cdot 2 (n-1)^2 \left(I_1+\Psi(p)\, I_2\,\varepsilon^2 \right) + o(\varepsilon^2) \right) \\
&=& 2^n\nu_{2n-1} \left( 4t(n-1)I_1 + \left(\mu^t(\omega)(p)\, J_0 + 4t(n-1)\Psi(p)\, I_2\right)\varepsilon^2 + o(\varepsilon^2) \right) .
\end{eqnarray*}

\smallskip
\noindent{\itshape Step 8. Taylor expansion of the functional on the test functions.}
We now put everything together. We use the elementary identity
$$ \frac{n_0+n_2\varepsilon^2+o(\varepsilon^2)}{d_0+d_2\varepsilon^2+o(\varepsilon^2)} = \frac{n_0}{d_0} + \left(\frac{n_2}{d_0}-\frac{n_0\,d_2}{d_0^2}\right)\varepsilon^2+o(\varepsilon^2) . $$
In our case, from Steps 4 and 7, we have
\begin{eqnarray*}
n_0 &:=& 2^n\nu_{2n-1} \cdot 4t(n-1)I_1 , \\
n_2 &:=& 2^n\nu_{2n-1} \cdot \left(\mu^t(\omega)(p)\, J_0 + 4t(n-1)\Psi(p)\, I_2\right) , \\
d_0 &:=& 2^{n-1}\left(\nu_{2n-1}I_0\right)^{\frac{n-1}{n}} , \\
d_2 &:=& d_0\,\Phi(p) .
\end{eqnarray*}
Therefore, recalling also \eqref{eq:sigma-I0}:
\begin{eqnarray*}
\lefteqn{ {\rm EH}^t(u_\varepsilon) } \\
&=& \frac{\int_X \left(u_\varepsilon^2\,\mu^t(\omega) + \frac{2t}{n-1}|\nabla u_\varepsilon|^2\right) \frac{\omega^n}{n!}}{\|u_\varepsilon\|_{L^{2^*}}^2} \\
&=& (\nu_{2n-1}I_0)^\frac{1}{n} \cdot \left( 8nt + \left( 4\cdot\frac{2n-1}{n-2}\mu^t(\omega)(p) + 8t\cdot\frac{n(n+1)}{n-2}\Psi(p) - 8nt\Phi(p) \right)\,\varepsilon^2 + o(\varepsilon^2) \right) \\
&=& \frac{2t}{(n-1)\sigma_{2n}} \\
&& + \frac{1}{2n(n-2)\sigma_{2n}} \left( 2\cdot\frac{2n-1}{n-1}\mu^t(\omega)(p) - \frac{2t}{n-1} \cdot 2n \cdot \left( (n-2)\Phi(p) - (n+1)\Psi(p) \right) \right)\,\varepsilon^2 \\
&& + o(\varepsilon^2) \\
&=& t\Lambda_n + \frac{1}{2n(n-2)\sigma_{2n}} \left( 2\cdot\frac{2n-1}{n-1}\mu^t(\omega)(p) - \frac{2t}{n-1} \cdot \mathrm{scal}(g_{\omega}) \right)\,\varepsilon^2 + o(\varepsilon^2) ,
\end{eqnarray*}
where we used $\Lambda_n:=\frac{2}{(n-1)\,\sigma_{2n}}$ and $2n\big((n-2)\Phi-(n+1)\Psi\big) = \mathrm{scal}(g_{\omega})$.

\smallskip
\noindent{\itshape Step 9. Universal upper bound for the deformed Yamabe invariants.}
Letting $\varepsilon\to 0$ in the previous expression, we get $\mathrm{EH}^t(u_\varepsilon) \to t\Lambda_n$, and therefore
$$ \lambda^t_X(\{\omega\}) \leq \inf_{\varepsilon>0} \mathrm{EH}^t(u_\varepsilon) \leq t\Lambda_n . $$
Note that this is not obvious a priori, since the test functions are constructed using specific holomorphic normal coordinates and do not coincide in general with those used by Aubin in \cite{aubin-1976}.

\smallskip
\noindent{\itshape Step 10. Curvature condition for strict inequality.}
We have the identity
$$ \frac{2(2n-1)}{n-1}\mu^t(\omega) - \frac{2t}{n-1} \mathrm{scal}(g_\omega) = -\frac{2(t-2n+1)}{n-1}\mu^0(\omega) , $$
so that the condition $(t-2n+1)\mu^0(\omega)(p)>0$ at some point $p\in X$ ensures that the correction term in Step 8 is strictly negative, giving
$$ \lambda^t_X(\{\omega\}) < t\Lambda_n . $$

\smallskip
\noindent{\itshape Step 11. Subcritical existence.}
We now show that this condition is sufficient to guarantee the existence of a smooth solution. We follow the argument as in \cite[Section 4]{lee-parker}.
Up to rescaling $\omega$ by a positive constant, which changes
neither the sign nor the vanishing of $\mu^t(\omega)$, nor the validity of the
hypothesis $\lambda^t_{2^*}<t\Lambda_n$, we may assume $\mathrm{Vol}(X,\omega)=1$.
For $2<s<2^*$, the functional $\mathrm{EH}^t_s$ defined in \eqref{eq:EH-t-p} admits a minimizer by Remark~\ref{rmk:EH-p}. Let $u_s$ be a smooth positive minimizer of $\mathrm{EH}^t_s$, normalized so that $\|u_s\|_{L^s}=1$. More precisely, $u_s$ satisfies
\begin{equation}\label{eq:subcritical}
\frac{2t}{n-1}\,\Delta u_s + \mu^t(\omega)\, u_s = \lambda^t_s \, u_s^{s-1} ,
\end{equation}
where
$$ \lambda^t_s := \inf_{u>0} \mathrm{EH}^t_s(u) . $$
We claim that one can extract a subsequence of $\{u_s\}_s$ converging as $s \to 2^*$. The key point is to establish suitable a priori estimates for equation \eqref{eq:subcritical}. Fix $\delta > 0$. From \eqref{eq:subcritical}, we obtain
\begin{eqnarray*}
\lambda^t_s \int_X u_s^{s + 2\delta} \,\frac{\omega^n}{n!}
&=& \int_X u_s^{1 + 2\delta} \left( \frac{2t}{n-1}\,\Delta u_s + \mu^t(\omega)\, u_s \right) \,\frac{\omega^n}{n!} \\
&=& \frac{2t}{n-1} \int_X u_s^{1 + 2\delta} \Delta u_s \,\frac{\omega^n}{n!} + \int_X \mu^t(\omega)\, \left(u_s^{1 + \delta}\right)^2 \,\frac{\omega^n}{n!} \\
&=& \frac{2t}{n-1} (1 + 2\delta) \, \int_X u_s^{2\delta}\, |d u_s|^2 \, \frac{\omega^n}{n!} + \int_X \mu^t(\omega)\, \left(u_s^{1 + \delta}\right)^2 \,\frac{\omega^n}{n!} .
\end{eqnarray*}
Setting
$$ w_s := u_s^{1+\delta} , $$
this gives
$$ \int_X\left(\lambda^t_s\,u_s^{s-2}-\mu^t(\omega)\right)w_s^2 \,\frac{\omega^n}{n!} = \frac{2t(1+2\delta)}{(1+\delta)^2(n-1)} \int_X|dw_s|^2_g \,\frac{\omega^n}{n!} . $$

As a consequence of the Sobolev inequality, as stated for example in \cite[Theorems~2.3 and~3.3]{lee-parker}, for any $\varepsilon>0$, there exists $C_\varepsilon>0$ such that
\begin{eqnarray*}
\| w_s \|_{L^{2^*}}^2
&\leq& (1+\varepsilon)\,\sigma_{2n}\,\|\nabla w_s \|_{L^2}^2 + C_\varepsilon \, \| w_s \|_{L^2}^2 \\
&=& (1+\varepsilon)\,\sigma_{2n}\, \frac{(1+\delta)^2(n-1)}{2t(1+2\delta)} \int_X\left(\lambda^t_s\,u_s^{s-2}-\mu^t(\omega)\right) w_s^2 \,\frac{\omega^n}{n!} + C_\varepsilon \, \| w_s \|_{L^2}^2 .
\end{eqnarray*}

When $\lambda_s\leq 0$, we immediately obtain the estimate
$$ \| w_s \|_{L^{2^*}}^2 \leq C'_{\varepsilon}\, \| w_s \|_{L^2}^2 , $$
where the constant $C'_{\varepsilon}$ depends only on $\varepsilon$ and $(X,\omega)$, but is independent of $s$.

We claim that the same uniform estimate holds also in the case $\lambda_s > 0$, provided that $\lambda^t_{2^*} < t\Lambda_n = \frac{2t}{(n-1)\sigma_{2n}}$.
Indeed, by H\"older's inequality we obtain
$$ \int_X u_s^{s-2}\,w_s^2\,\frac{\omega^n}{n!}
\leq \|u_s^{s-2}\|_{L^{\frac{2^*}{2^*-2}}} \cdot \|w_s^2\|_{L^{\frac{2^*}{2}}}
= \|u_s\|^{s-2}_{L^{\frac{2^*(s-2)}{2^*-2}}} \cdot \|w_s\|^2_{L^{2^*}}
\leq \|w_s\|^2_{L^{2^*}} . $$
Here the last inequality follows from the fact that $s<2^*$, which implies $\frac{2^*(s-2)}{2^*-2}<s$, together with the monotonicity of $L^s$ norms on a compact manifold of unit volume.
Therefore we obtain
\[
\|w_s\|_{L^{2^*}}^2 \;\le\; (1+\varepsilon)\,\sigma_{2n}\,\frac{(1+\delta)^2(n-1)}{2t(1+2\delta)}\,\lambda^t_s\,\|w_s\|_{L^{2^*}}^2 \;+\; C'_\varepsilon\,\|w_s\|_{L^2}^2,
\]
where $C'_\varepsilon$ denotes a possibly larger constant than $C_\varepsilon$, depending only on $(X, \omega)$, $n$, $t$, $\delta$, $\varepsilon$ but independent of $s$.
The crucial observation is the following. First, note that the functionals $\mathrm{EH}^t_s$ are related, for varying $s$, by $\mathrm{EH}^t_{s'}(u) = \|u\|_{L^s}^2\cdot\|u\|_{L^{s'}}^{-2} \cdot \mathrm{EH}^t_s(u)$.
In particular, if $\lambda^t_s<0$ for some $s\in[2,2^*]$, then $\lambda^t_s<0$ for every $s\in[2,2^*]$.
This allows us to apply Aubin's continuity argument \cite[Lemma 4.3]{lee-parker}, which shows that $|\lambda^t_s|$ is non-increasing as a function of $s\in[2,2^*]$. Moreover, since $\lambda^t_{2^*}\ge0$, the function $s\mapsto\lambda^t_s$ is continuous from the left. It is precisely here that the hypothesis $\lambda^t_{2^*}<t\Lambda_n$ comes into play: it provides a strict margin in the coefficient inequality, which is what prevents the sequence $u_s$ from concentrating like a bubble as $s\to(2^*)^-$. Indeed, when $\lambda^t_{2^*}<2t/((n-1)\sigma_{2n})$, for $s$ sufficiently close to $2^*$ we also have $\lambda^t_s<2t/((n-1)\sigma_{2n})$, and we can choose $\varepsilon>0$ and $\delta_0>0$ sufficiently small such that
\[
(1+\varepsilon)\,\sigma_{2n}\,\frac{(1+\delta_0)^2(n-1)}{2t(1+2\delta_0)}\,\lambda^t_s \;<\; (1+\varepsilon)\,\frac{(1+\delta_0)^2}{1+2\delta_0} \;<\; 1 .
\]
Then, shrinking $\delta_0$ further if necessary so that also $2(1+\delta_0)<s$, and using Jensen's inequality (recall $\mathrm{Vol}(X,\omega)=1$), we get
\begin{align*}
\|w_s\|_{L^{2^*}}^2 &\le C''\,\|w_s\|_{L^2}^2 = C''\,\|u_s^{1+\delta_0}\|_{L^2}^2 = C''\,\|u_s\|_{L^{2+2\delta_0}}^{2+2\delta_0} \le C''\,\|u_s\|_{L^s}^{2+2\delta_0} = C'',
\end{align*}
for some constant $C''>0$ depending only on $(X,\omega)$, $n$, $t$, $\varepsilon$ and $\delta_0$, but independent of $s$. Therefore $w_s$ is uniformly bounded in $L^{2^*}$.

Now that we have a uniform bound on $\|w_s\|_{L^{2^*}}$, and recalling $w_s=u_s^{1+\delta_0}$, we get $\|u_s\|_{L^{2^*(1+\delta_0)}}^{1+\delta_0}=\|w_s\|_{L^{2^*}}\leq\sqrt{C''}$, i.e.\ $u_s$ is uniformly bounded in $L^{2^*(1+\delta_0)}$. This is an improvement over the a priori bound $u_s\in L^{2(1+\delta_0)}$, coming from the normalization $\|u_s\|_{L^s}=1$ together with $2(1+\delta_0)<s$ and the monotonicity of $L^s$ norms (recall
$\mathrm{Vol}(X,\omega)=1$).

We now conclude with a standard bootstrap argument, which no longer requires the small constant condition. We now regard \eqref{eq:subcritical} as a linear equation for $u_s$,
\[
\Delta u_s = \frac{n-1}{2t}\Big(\lambda^t_s\,u_s^{s-1}-\mu^t(\omega)\,u_s\Big),
\]
whose right-hand side we already control: since $u_s$ is uniformly bounded in $L^{2^*(1+\delta_0)}$, the term $u_s^{s-1}$ is uniformly bounded in $L^{2^*(1+\delta_0)/(s-1)}$, where the exponent $2^*(1+\delta_0)/(s-1)$ stays in a compact subset of $(1,+\infty)$ as $s$ ranges over $(2,2^*)$ (and hence so does every subsequent exponent produced by the bootstrap argument below).
A standard elliptic bootstrap (Calder\'on-Zygmund estimates followed by Sobolev embedding, iterated finitely many times as the integrability exponent increases to reach $n$, with Calder\'on-Zygmund constants that stay uniformly bounded for the integrability exponent ranging over any fixed compact subset of $(1,+\infty)$) then yields a uniform $L^\infty$ bound on $\{u_s\}_{2<s<2^*}$.
Coming back to \eqref{eq:subcritical} and applying Schauder estimates gives a uniform $C^{2,\alpha}$ bound. By Ascoli-Arzel\`a, a subsequence $u_{s_k}\to u$ converges in $C^2$ as $s_k\to(2^*)^-$. Passing to the limit shows that $u\geq0$ solves \eqref{eq:t-DY} with $c=\lim_k\lambda^t_{s_k}$, and $u\not\equiv0$ since $\|u\|_{L^{2^*}}=\lim_k\|u_{s_k}\|_{L^{2^*}}>0$. Finally, the strong maximum principle upgrades this to $u>0$ everywhere, concluding the proof.
\end{proof}

\begin{remark}
As noted after \eqref{eq:conf-mu0}, the sign of $\mu^0$ at a point $p\in X$ is conformally invariant.
\end{remark}

\begin{remark}\label{rmk:hopf}
The condition $(t-2n+1)\,\mu^0(\omega)(p)>0$ is far from necessary. For example, consider the Hopf manifold
$$ X = (\mathbb{C}^n \setminus\{0\}) / \mathbb{Z} , $$
where the action of $\mathbb{Z}$ is generated by $z\mapsto \alpha z$ for some fixed $\alpha\in\mathbb{C}$ with $0<|\alpha|<1$. Endow $X$ with the Hermitian metric
$$ \omega = |z|^{-2} \delta_i^j \,\sqrt{-1}\, dz^i\wedge d\bar{z}^j , $$
where $\{z^i\}_i$ are the standard coordinates in $\mathbb C^n$ and $\omega_0:=\delta_i^j \,\sqrt{-1}\, dz^i\wedge d\bar{z}^j$ is the standard flat Hermitian structure on $\mathbb C^n$.
Note that $X$ is diffeomorphic to $S^1 \times S^{2n-1}$, the complex structure coincides with the transverse complex structure on $S^{2n-1}$ and sends the Reeb vector field of $S^{2n-1}$ to the generator of $S^1$, and the underlying Riemannian metric is the product of round metrics on $S^1$ and $S^{2n-1}$, whose radii depend on $\alpha$. More precisely, note that the underlying Riemannian metric of $\omega_0$ on $\mathbb C^n$ is $g_0=\omega_0(\cdot,J\cdot)=2\,|dz|^2$, i.e.\ twice the Euclidean metric on $\mathbb R^{2n}$. Using polar coordinates, $z=tu$ with $t=|z|>0$ and $u\in S^{2n-1}$ the unit sphere in $\mathbb C^n$, we get $g=2\frac{|dz|^2}{|z|^2}=2dt^2+2g_{S^{2n-1}}$, where $g_{S^{2n-1}}$ denotes the round metric on $S^{2n-1}$ of radius $1$. Substituting $s:=\log t$, we get $g=2ds^2+2g_{S^{2n-1}}$. The deck transformation $z \mapsto \alpha z$ acts on $s$ by translation $s \mapsto s + \log|\alpha|$. Set $L:=-\log|\alpha|$. Therefore $X \simeq \left( \mathbb R / L \mathbb Z \right) \times S^{2n-1}$ as a Riemannian product, where the $S^{2n-1}$ factor carries the metric $2\,g_{S^{2n-1}}$, i.e.\ a round sphere of radius $r=\sqrt{2}$, and the $S^1$ factor has length $\int_0^L\sqrt2\,ds=\sqrt2\,L$ with respect to $g$, hence radius $\rho=\frac{\sqrt2\,L}{2\pi}=\frac{\sqrt2}{2\pi}\,\big|\log|\alpha|\big|$.

The Hermitian structure is invariant under the transitive action of $U(1) \times U(n)$, so $\mu^t(\omega)$ is constant. A direct computation gives
$$ \Theta_{i\bar{j}k\bar{\ell}} = |z|^{-6}\Big(|z|^2\delta_{i}^{j} -\bar{z}^{i}z^{j}\Big)\delta_{k}^{\ell} , \quad T_{ij}^k = -|z|^{-2}\Big(\bar{z}^{i}\delta_{j}^{k} -\bar{z}^{j}\delta_{i}^{k}\Big) , $$
and therefore
$$
\mathrm{scal}^{\mathrm{Ch}}(\omega) = 2n(n-1) , \quad
d^*\theta = 0 , \quad
|T|^2 = 4(n-1) ,
$$
so that the $t$-deformed scalar curvature is
$$ \mu^t(\omega) = t\,(n-1) . $$
In particular, $\mu^0(\omega)=0$ and $\mu^t(\omega)$ is constant for every $t$, as already noticed, and positive for $t>0$. For completeness, we also compute
$$
\mathrm{EH}^t(\omega) = t\,\Lambda_n\frac{n-1}{2n}\frac{\mathrm{Vol}(S_\rho^1 \times S_r^{2n-1})^{\frac{1}{n}}}{\mathrm{Vol}(S^{2n})^{\frac{1}{n}}} ,
$$
where we recall that the radius of the spheres $r=\sqrt{2}$ and $\rho=\frac{\sqrt2}{2\pi}\big|\log|\alpha|\big|$ depend on the parameter $\alpha$ in the construction, while $S^{2n}$ is the reference unit sphere used to normalize the sharp Sobolev constant $\Lambda_n$.
Recall also that
$$ \mathrm{Vol}(S_R^m ) = \frac{2\pi^{\frac{m+1}{2}}}{\Gamma(\frac{m+1}{2})} \cdot R^m , $$
Note that the coefficient
\begin{eqnarray*}
\gamma_{n,\alpha}
&:=& \frac{n-1}{2n}\left(\frac{\mathrm{Vol}\big(S^1_{\frac{\sqrt2}{2\pi}\big|\log|\alpha|\big|}\times S^{2n-1}_{\sqrt2}\big)}{\mathrm{Vol}(S^{2n})}\right)^{1/n} \\
&=& \frac{n-1}{4n}\left(n\binom{2n}{n}\,\big|\log|\alpha|\big|\right)^{1/n}
\end{eqnarray*}
is not necessarily less than $1$.
For example, for $n=3$ and $\alpha=\frac{1}{2}$, we get $\gamma_{n=3,\alpha=\frac{1}{2}}\approx0.57744$, but $\gamma_{n=15,\alpha=\frac{1}{4}}\approx1.0044$.
For fixed $\alpha$, the asymptotic behaviour of $\gamma(n,\alpha)$ as $n\to +\infty$ is
$$ \gamma(n,\alpha) = 1 + \frac{1}{n}\left(\frac{1}{2}\ln\frac{n}{\pi} \;+\; \ln\big|\log|\alpha|\big| \;-\; 1\right) \;+\; o\!\left(\frac{\ln n}{n}\right), \qquad n\to\infty . $$
\end{remark}

\begin{remark}
When $t = 2n-1$, corresponding to the classical Yamabe problem, the quantity $(t-2n+1)\,\mu^0(\omega)$ vanishes identically, so this criterion gives no information. In the classical Yamabe problem, one instead examines the next coefficient in the Taylor expansion, which depends on the Weyl tensor.
\end{remark}

\begin{remark}
Beyond the case of $\mathbb CP^1\simeq S^2$, it is still not clear whether the equality $\lambda^t_X(\{\omega\})=t\,\Lambda_n$ can ever be realized. (Note also that the existence of complex structures on $S^6$ is itself an open problem.)
\end{remark}

As a consequence of Theorem~\ref{thm:deformed-yamabe-existence}, the situation in the globally conformally K\"ahler setting is completely understood.

\begin{corollary}
Let $(X,\omega_K)$ be a compact K\"ahler manifold of complex dimension $n\geq3$, and fix $t>2n-1$. Then there exists a Hermitian metric conformal to $\omega_K$ with constant $t$-deformed scalar curvature.
\end{corollary}

\begin{proof}
Up to homothety, assume that $\mathrm{Vol}(X,\omega_K)=1$. Since $\omega_K$ is K\"ahler, we have $\theta=T=0$, and therefore
$$ \mu^0(\omega_K) = \mathrm{scal}^{\mathrm{Ch}}(\omega_K) - \frac{n}{n-1} \, d^*\theta - \frac{n}{2} \, |T|^2 = \mathrm{scal}^{\mathrm{Ch}}(\omega_K) . $$
If there exists a point $p\in X$ where $\mathrm{scal}^{\mathrm{Ch}}(\omega_K)(p)>0$, then $(t-2n+1)\mu^0(\omega_K)(p)>0$ since $t>2n-1$, and Theorem~\ref{thm:deformed-yamabe-existence} applies. Otherwise, $\mathrm{scal}^{\mathrm{Ch}}(\omega_K)\leq 0$ everywhere, so
$$ \lambda^t_X(\{\omega_K\}) \leq \mathrm{EH}^t(\omega_K) = \int_X \mathrm{scal}^{\mathrm{Ch}}(\omega_K)\,\frac{\omega_K^n}{n!} \leq 0 < t\Lambda_n , $$
and Theorem~\ref{thm:deformed-yamabe-existence} applies again.
\end{proof}

\begin{remark}
Several natural directions are not pursued here. One could study the critical points of the $t$-deformed Einstein-Hilbert functional on the full space of Hermitian metrics on $X$, rather than on a single conformal class. This leads to the $t$-deformed Yamabe invariant
$$ \sigma^t_{\mathbb{C}}(X) := \sup_{\omega \text{ Hermitian}} \lambda^t_{X}(\{\omega\}) , $$
a function of the parameter $t$. Note that, even for $t=2n-1$, this does \emph{not} recover the classical Yamabe invariant, since the supremum ranges only over Hermitian metrics on $X$, and not over all Riemannian metrics. It would be interesting to investigate the relation between $\sigma^t_{\mathbb{C}}(X)$ and the underlying complex geometry, in the spirit of \cite{lebrun-CAG, albanese-lebrun}. One could also restrict the supremum to special classes of Hermitian metrics, K\"ahler or non K\"ahler, as done in the Sasakian setting in \cite{lahdili-legendre-scarpa}, where a related invariant is linked to the $K$-semistability of K\"ahler cones.
\end{remark}

\subsection{Examples: cohomogeneity one metrics on Hirzebruch surfaces}

We conclude by showing a construction for metrics with positive constant deformed scalar curvature, beyond the locally homogeneous examples. The construction follows the cohomogeneity-one method that Page \cite{page} and B\'erard Bergery \cite{berardbergery} used to produce the first non-homogeneous Einstein metrics with positive scalar curvature, and that was adapted to the Hermitian non-K\"ahler setting in \cite{angella-pediconi} to construct non-homogeneous metrics with constant Chern scalar curvature. We specialize to Hirzebruch surfaces and to the momentum case $t=2n=4$.

For an integer $m>0$, the $m$-th Hirzebruch surface $X_m$ is the total space of the $\mathbb CP^1$-bundle $\pi_m\colon X_m\to\mathbb CP^1$ associated with the $U(1)$-principal bundle $s_m\colon\Sigma_m=S^3/\mathbb{Z}_m\to\mathbb CP^1$ of Euler class $-\frac{m}{2\pi}[\omega_{\rm FS}]$, where $\omega_{\rm FS}$ is normalized so that $\mathrm{Ric}(g_{\rm FS})=2g_{\rm FS}$. It is compact and simply connected, and carries a holomorphic action of $U(2)$ of cohomogeneity one, with regular part diffeomorphic to $\Sigma_m\times(0,L)$ and two singular orbits, both biholomorphic to $\mathbb CP^1$.

Fix a principal connection $\vartheta_m$ on $\Sigma_m$ with $d\vartheta_m=-m\,s_m^*\omega_{\rm FS}$, and for $\nu,\kappa>0$ let $\tilde{g}(\nu,\kappa)$ be the $U(2)$-invariant metric on $\Sigma_m$ whose vertical part is the round circle of length $\frac{2\pi}{m}\nu$, whose horizontal distribution is $\ker\vartheta_m$, and which makes $s_m\colon(\Sigma_m,\tilde{g}(\nu,\kappa))\to(\mathbb CP^1,\kappa^2 g_{\rm FS})$ a Riemannian submersion with totally geodesic fibres. Given smooth positive functions $f,p\colon(0,L)\to\mathbb{R}$, the formula
$$ g(f,p) := dr^2 + \tilde{g}(f(r),p(r)) $$
defines a $U(2)$-invariant Hermitian metric on the regular part, extending smoothly to $X_m$ if and only if $f$ extends to a smooth odd function with $f(L+r)=-f(L-r)$ and $f'(0)=1=-f'(L)$, and $p$ extends to a smooth even function with $p(L+r)=p(L-r)$. By \cite[Theorem A]{angella-pediconi}, the metric $g(f,p)$ is always globally conformally K\"ahler, and it is K\"ahler if and only if $2p p'+mf=0$. Its Chern scalar curvature and the codifferential of its Lee form are computed in \cite{angella-pediconi} to be
\begin{eqnarray*}
\mathrm{scal}^{\mathrm{Ch}}(r) &=& -2\frac{f''(r)}{f(r)} - 2\frac{p''(r)}{p(r)} + 2\left(\frac{p'(r)}{p(r)}-\frac{f'(r)}{p(r)}\right)\frac{p'(r)}{p(r)} \\
&& + \frac{4}{p(r)^2} + 2m\left(f'(r)+f(r)\frac{p'(r)}{p(r)}\right)\frac{1}{p(r)^2} , \\
d^*\theta(r) &=& \frac{1}{p(r)^2}\left( \big(2p(r)p'(r)+mf(r)\big)\frac{f'(r)}{f(r)} + 2p(r)p''(r) + 2p'(r)^2 + mf'(r) \right) .
\end{eqnarray*}

\begin{theorem}[{\cite{angella-pediconi-scarpa-spotti-windare}}]\label{thm:pediconi-metrics}
For every integer $m>0$, the Hirzebruch surface $X_m$ admits a $U(2)$-invariant, globally conformally K\"ahler, non-K\"ahler metric with positive constant momentum $\mu^{t=4}$.
\end{theorem}

\begin{proof}[Sketch]
Consider the ansatz
$$ f_\varphi(r) := -\frac{1}{m}\varphi'(r) , \qquad p_\varphi(r) := \sqrt{\tfrac{5}{2}\,\varphi(r)} , $$
for a smooth, positive, decreasing function $\varphi\colon(0,L)\to\mathbb{R}$. Then $2p_\varphi p_\varphi'+mf_\varphi=\frac{3}{2}\varphi'<0$, so $g(f_\varphi,p_\varphi)$ is never K\"ahler. Substituting the ansatz into the formulas above, the equation $\mu^4(g(f_\varphi,p_\varphi))=c$ becomes the third-order ordinary differential equation
$$ 10\,\varphi^2\,\frac{\varphi'''}{\varphi'} + 2\,\varphi\varphi'' - 3\,(\varphi')^2 + 5c\,\varphi^2 - 8\,\varphi = 0 , $$
and the smoothness conditions translate into the boundary conditions $\varphi(0)=1$, $\varphi(L)=k$, $\varphi'(0)=\varphi'(L)=0$, $\varphi''(0)=-q=-\varphi''(L)$, for some $0<k<1$ and $q>0$. The key trick is the substitution $\varphi'(r)=-\sqrt{u(\varphi(r))}$: under the change of variable $t=\varphi(r)$, the ordinary differential equation becomes linear,
$$ 5t^2u''(t) + tu'(t) - 3u(t) + 5c\,t^2 - 8t = 0 , $$
with general solution
$$ u_{a,b,c}(t) = a\,t^{\frac{2-\sqrt{19}}{5}} - 4t + b\,t^{\frac{2+\sqrt{19}}{5}} - \frac{5c}{9}\,t^2 , \qquad a,b,c\in\mathbb{R} , $$
and the boundary conditions become $u(1)=u(k)=0$, $u'(1)=-2m$, $u'(k)=2m$. Imposing the first three determines $a(k)$, $b(k)$, $c(k)$ explicitly; the fourth becomes an equation $\alpha(k)/\beta(k)=0$ in the remaining parameter $k$. An asymptotic analysis at $k=1$ (where $\alpha$ vanishes to third order with $\alpha'''(1)=-\frac{36\sqrt{19}}{25}m<0$) and at $k=0$ (where $\alpha(0)<0$) shows the existence of $k_*\in(0,1)$ with $\alpha(k_*)=0$, while $\beta>0$ on $(0,1)$. Evaluating the linear ordinary differential equation at an interior maximum point $t_0$ of the resulting solution $u_*$ gives $5c_*t_0^2=-5t_0^2u_*''(t_0)+3u_*(t_0)+8t_0>8t_0$, whence $c_*>0$; a similar maximum-principle argument at a hypothetical interior minimum shows $u_*>0$ on $(k_*,1)$. Inverting $\varphi_*(r)$ from $\varphi_*'=-\sqrt{u_*(\varphi_*)}$ then produces a smooth solution of the original boundary value problem, and the metric $g(f_{\varphi_*},p_{\varphi_*})$ is the desired one, with $\mu^4=c_*>0$.
\end{proof}

\begin{footnotesize}
\noindent{\bfseries Acknowledgements.}
These notes were prepared for the course delivered by the author at the CIME Summer School \emph{Cohomological properties of almost complex manifolds and special structures}, held in Cetraro (CS), July 13--17, 2026. The author warmly thanks the Scientific Directors of the School, Anna Maria Fino, Adriano Tomassini, and Scott O. Wilson, for the kind invitation and for planning such a stimulating and enjoyable event, and all the participants for contributing to a lively and positive atmosphere.
The author is deeply grateful to his collaborators, Simone Calamai, Francesco Pediconi, Carlo Scarpa, Cristiano Spotti, and Joshua Windare, whose joint works are the subject of these notes, and with whom collaborating was always as joyful and caring as it was mathematically enriching, as well as to all other (past and future) collaborators for their continued support along the way.

\smallskip

\noindent{\bfseries Funding.}
The author is partially supported by GNSAGA of INdAM.

\smallskip

\noindent{\bfseries Competing interests.}
The author is a member of the Scientific Committee and serves as Scientific Secretary of CIME.

\smallskip

\noindent{\bfseries Tool and computational resource disclosure.}
Large language models were used to assist with English proofreading, notational consistency, and \LaTeX\ transcription of results established by the author and collaborators. In some instances, they were also used to cross-check computations and gather results from the literature. All outputs were verified by the author, who takes full responsibility for the content. This disclosure is made in the spirit of the Leiden Declaration and the UNESCO Recommendation on Open Science.
\end{footnotesize}

\bibliographystyle{alpha}
\bibliography{biblio}

\end{document}